\documentclass[11pt]{amsart} 

\usepackage[utf8]{inputenc} 

\usepackage{lipsum}

\usepackage{geometry} 
\usepackage{graphicx} 

\usepackage[english]{babel}

\usepackage{booktabs} 
\usepackage{array} 
\usepackage{paralist} 
\usepackage{verbatim} 
\usepackage{subfig} 

\usepackage{caption}

\usepackage{amsfonts,amsmath,amssymb,bbm,amsthm,amsbsy,amscd}
\usepackage{tikz-cd}
\usepackage{mathtools}
\usepackage{pinlabel}

\usepackage{color}
\usepackage{pdfcolmk}

\usepackage{fancyhdr} 
\usepackage{tikz}
\usepackage{tikz-cd}

\usepackage{IEEEtrantools}

\newenvironment{equ*}[1]{\begin{IEEEeqnarray*}{#1}}{\end{IEEEeqnarray*}}

\usepackage[pdfborder={0 0 0 [0 0 ]},bookmarksdepth=3,bookmarksopen=true]{hyperref}
\usepackage[capitalize]{cleveref}
\usepackage[ps,all,arc,rotate]{xy}
\usepackage{bbm} 

\hypersetup{
    colorlinks=true,
    linkcolor=black,
    citecolor=black,
    filecolor=black,
    urlcolor=black,
}

\usepackage{cleveref}

\newenvironment{frcseries}{\fontfamily{frc}\selectfont}{}

\makeatletter                                                   
\newtheorem*{rep@theorem}{\rep@title}
\newcommand{\newreptheorem}[2]{%
\newenvironment{rep#1}[1]{%
 \def\rep@title{#2 \ref{##1}}%
 \begin{rep@theorem}}%
 {\end{rep@theorem}}}
\makeatother

\newtheorem{thm}{Theorem}[section]
\newtheorem{lemma}[thm]{Lemma}

\crefname{lemma}{Lemma}{Lemmata}
\newtheorem{prop}[thm]{Proposition}
\newtheorem{corr}[thm]{Corollary}

\newtheorem*{conical_circle_theorem}{Conical Circle Theorem~\ref{theorem:conical_action}}
\newtheorem*{countable_orbit_theorem}{Countable Orbit Theorem~\ref{theorem:countable_orbit}}
\newtheorem*{rigidity_theorem}{Rigidity Theorem~\ref{thm: rigidity}}
\newtheorem*{homology_theorem}{Homology Theorem~\ref{thm: H_2}}
\newtheorem*{thm*}{Theorem}
\newtheorem*{lemma*}{Lemma}
\newtheorem*{prop*}{Proposition}
\newtheorem*{corr*}{Corollary}
\newtheorem*{claim*}{Claim}

\theoremstyle{remark}
\newtheorem{rmk}[thm]{Remark}

\newtheorem{quest}[thm]{Question}

\newtheorem*{rmk*}{Remark}
\newtheorem*{conj*}{Conjecture}
\newtheorem*{quest*}{Question}

\theoremstyle{definition}
\newtheorem{defn}[thm]{Definition}
\newtheorem{exmp}[thm]{Example}

\newtheorem*{defn*}{Definition}
\newtheorem*{exmp*}{Example}

\newreptheorem{theorem}{Theorem}
\newreptheorem{corollary}{Corollary}
\newreptheorem{proposition}{Proposition}

\newcommand{\R}{\mathbb{R}}
\def\C{\mathbb{C}}

\newcommand{\D}{\mathbb{D}}
\newcommand{\Q}{\mathbb{Q}}
\newcommand{\Z}{\mathbb{Z}}
\newcommand{\N}{\mathbb{N}}

\newcommand{\Gammahat}{\widehat{\Gamma}}
\newcommand{\Omegahat}{\widehat{\Omega}}

\newcommand{\Hom}{\mathrm{Hom}}
\newcommand{\Homeorm}{\mathrm{Homeo}}
\newcommand{\MCG}{\mathrm{Mod}}

\newcommand{\rayend}{\mathrm{end}}
\newcommand{\PSL}{\mathrm{PSL}}
\newcommand{\eurm}{\mathrm{eu}}
\newcommand{\rot}{\mathrm{rot}}

\newcommand{\defeq}{\vcentcolon=}

\title{Big Mapping Class Groups and rigidity of the Simple Circle}

\author{Danny Calegari}
\address{Department of Mathematics\\ University of Chicago\\ Chicago, Illinois, USA}
\email[D.~Calegari]{dannyc@math.uchicago.edu}

\author{Lvzhou Chen}
\address{Department of Mathematics\\ University of Chicago\\ Chicago, Illinois, USA}
\email[L.~Chen]{lzchen@math.uchicago.edu}

\begin{document}

\begin{abstract}
Let $\Gamma$ denote the mapping class group of the plane minus a Cantor set. 
We show that every action of $\Gamma$ on the circle is either trivial
or semi-conjugate to a unique minimal action on the so-called {\em simple circle}.
\end{abstract}

\maketitle

\tableofcontents

\section{Introduction}

Let $\Gamma$ denote the mapping class group of the plane minus a Cantor set. This group
and its subgroups arise naturally in dynamics e.g.\/ \cite{Calegari_planar, Calegari_complex}
and has been studied in detail by Juliette Bavard \cite{Bavard_hyperbolic} and
Bavard--Walker \cite{Bavard_Walker_1, Bavard_Walker_2}. It is the best
understood and most studied example of a so-called `big' mapping class group; see
e.g.\/ \cite{Aramayona_Fossas_Parlier, Durham_Fanoni_Vlamis, Patel_Vlamis}. 

Calegari \cite{Calegari_planar} constructed a faithful action of $\Gamma$ on the circle
by homeomorphisms. Bavard--Walker \cite{Bavard_Walker_1} constructed a different
action which is more closely related to the action of $\Gamma$ on the Gromov boundary
of the {\em Ray graph}; see \cite{Bavard_hyperbolic}. These actions are both constructed
via hyperbolic geometry but they are in fact different. It is natural to ask if they
are related, and more generally to try to classify actions of $\Gamma$ on the circle.
In this paper we give such a classification: there is a unique minimal action on the
so-called {\em simple circle}, and every nontrivial action is semiconjugate to this one.
In other words for every nontrivial action of $\Gamma$ on the circle there is a unique
minimal compact nonempty invariant subset, and after collapsing complementary intervals
to points, the action on the quotient circle is conjugate to the action on the simple
circle (possibly up to a change of orientation).

The description of the simple circle is closely related to Bavard--Walker's {\em conical circle}. 
Our analysis of the dynamics in this case sharpens some of the results
of \cite{Bavard_Walker_1, Bavard_Walker_2} and gives new proofs of others.

\medskip

It is worth remarking that the analogue of our main result is known for
certain mapping class groups of finite type. Let $\Gamma_{g,1}$ denote the
mapping class group of a surface of genus $g$ with 1 marked point. This group
acts (by point-pushing) on the ideal boundary of the fundamental group of the
surface; one calls this the {\em geometric action}. Mann--Wolff \cite{Mann_Wolff}
show that every nontrivial action of $\Gamma_{g,1}$ on the circle is semiconjugate
to the geometric action when $g>2$. Their proof and ours are quite dissimilar, but
there is some conceptual overlap.

\medskip

{\noindent \bf Statement of Results.}
In \S~2--3 we introduce, following Bavard--Walker \cite{Bavard_Walker_1}, the action of
$\Gamma$ on the conical circle.

Let $\Omega$ denote the plane minus a Cantor set. The {\em conical cover} $\Omega_C$
is the covering space associated to the $\Z$ subgroup of $\pi_1(\Omega)$ generated
by a loop around infinity. $\Omega_C$ is conformal to $\D-0$ and the 
{\em conical circle} $S^1_C$ is identified with $\partial \D$. The group $\Gamma$
acts naturally on this circle by homeomorphisms.

Points in $S^1_C$, thought of as endpoints, correspond to bi-infinite geodesics in a hyperbolic structure on
$\Omega_C$ `starting' at $\infty$. Distinguished amongst these are the subset of
{\em simple} geodesics, those whose image under the covering projection is embedded 
in $\Omega$. The simple geodesics fall into 3 classes:
the {\em short rays} $R$, the {\em lassos} $L$ and the {\em long rays} $X$.

The main theorem we prove in this section relates the dynamics on $S^1_C$ to the
topology of these geodesics:

\begin{conical_circle_theorem}
The set $R\cup X$ is a Cantor set in $S^1_C$ and is the unique minimal set for the
action of $\Gamma$ on $S^1_C$. Every complementary interval contains a unique lasso,
and all lassos arise this way. The set of boundary points of $R\cup X$ is exactly
the set of long rays that spiral around some simple geodesic loop in $\Omega$.
\end{conical_circle_theorem}

\medskip

In \S~4 we define the simple circle $S^1_S$ as the quotient of $S^1_C$ obtained
by collapsing complementary gaps to $R\cup X$. The action of $\Gamma$ on $S^1_S$ is
minimal, and has both countable and uncountable orbits. The countable orbits have
the following characterization:

\begin{countable_orbit_theorem}
A point in $S^1_C$ has a countable orbit if and only if the associated geodesic
$\gamma$ stays inside some subsurface $\Sigma \subset \Omega$ of finite type.
\end{countable_orbit_theorem}

\medskip

In \S~5 we give some examples and non-examples of subgroups of $\Gamma$. A
countable group is a subgroup of $\Gamma$ if and only if it is circularly orderable.
The situation for uncountable groups is more complicated: the abstract
group $S^1$ embeds in $\Gamma$, whereas $\PSL_2(\R)$ does not.

\medskip

\S~6 is the heart of the paper, and contains the main result:

\begin{rigidity_theorem}
Any homomorphism $\Gamma \to \Homeorm^+(S^1)$ is either trivial, or is semiconjugate
(possibly up to a change of orientation) to the action on the simple circle $S^1_S$.
\end{rigidity_theorem}

The argument depends on an analysis of certain special subgroups of $\Gamma$.
If $r \in R$ is a simple ray in $\Omega$, we denote by $\Gamma_r$ the stabilizer of $r$
in $\Gamma$, and by $\Gamma_{(r)}$ the subgroup of $\Gamma_r$
represented by homeomorphisms which are the identity in a neighborhood of the endpoint
of $r$.

The subgroups $\Gamma_r$, $\Gamma_{(r)}$ and their properties are part of the
essential structure of $\Gamma$, and these structural results should be of independent
interest.

The skeleton of the argument proceeds as follows.

\begin{enumerate}
\item{for any action of $\Gamma$ on the circle, each $\Gamma_r$ has a global fixed point;}\label{item: first step}
\item{for any two rays $r,s$ with distinct endpoints, $\Gamma_r$ and $\Gamma_s$
together generate $\Gamma$;}
\item{thus (with some work) one shows that either there is a global fixed point for
$\Gamma$, or there is a $\Gamma$-equivariant injective map $P$ from $R$ to $S^1$; and finally}
\item{the circular order on $P(R)$ is rigid.}
\end{enumerate}

Identifying $P(R)$ in a given circle with $R$ in the simple circle gives the 
semiconjugacy. The proof of (\ref{item: first step}) uses bounded cohomology, 
and a homological vanishing argument due to Mather.

\medskip

Low-dimensional (co-)homology, especially of certain subgroups of $\Gamma$, plays a
significant role in the proof of our main theorem, and it is evidently important 
to try to understand the homology of $\Gamma$ and some related groups. 
In an appendix we compute the (co-)homology in dimension 2 of $\Gammahat$, 
the mapping class group of the sphere minus a Cantor set 
(which we denote $\Omegahat$). The group $\Gammahat$ is 
(uniformly) perfect, so $H_1(\Gammahat;\Z)=H^1(\Gammahat;\Z)=0$. The main
result in the appendix is:

\begin{homology_theorem}
$H_2(\Gammahat;\Z)=\Z/2$ and $H^2(\Gammahat;\Z)=0$.
\end{homology_theorem}

\section{The conical circle}

Let $\Omega$ denote the plane minus a Cantor set.
This surface admits many different hyperbolic structures. The notions of conical cover and 
conical circle introduced below do not depend on the choice of hyperbolic structures.
For concreteness, choose a
Schottky subgroup of $\PSL_2(\C)$, and consider the conformal structure on the sphere minus
its limit set, then remove some point and hyperbolize the result. This hyperbolic structure
has the nice property that away from any neighborhood of the puncture (`infinity') the
geometry is precompact in the space of pointed metric spaces, though we do not make use
of this fact in the sequel.

The {\em conical cover} $\Omega_C$ is the covering space of $\Omega$ associated to the
$\Z$ subgroup of $\pi_1(\Omega)$ generated by a loop around infinity; See Figure \ref{fig: conical}. 
Geometrically, $\Omega_C$ is conformally equivalent to the punctured unit disk. The {\em conical circle}
$S^1_C$ is the boundary of the unit disk, and it may be identified in a natural way with
the space of proper bi-infinite geodesics in $\Omega_C$ with one end going out the
puncture. Informally, we call these {\em geodesic rays}. The identification with $S^1_C$
topologizes the space of such geodesic rays.

\begin{figure}
	\labellist
	\small 
	
	\pinlabel $\tilde{\infty}$ at 61 185
	\pinlabel $\Omega_C$ at 17 50
	\pinlabel $S^1_C$ at 105 -5
	
	\pinlabel $\infty$ at 353 185
	\pinlabel $\Omega$ at 300 40
	\endlabellist
	\centering
	\includegraphics[scale=0.7]{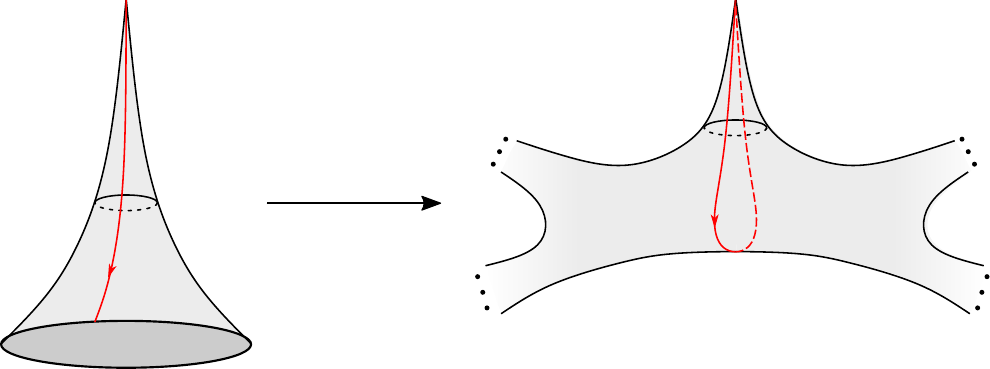}
	\caption{The conical cover with a geodesic ray projecting to a lasso}\label{fig: conical}
\end{figure}
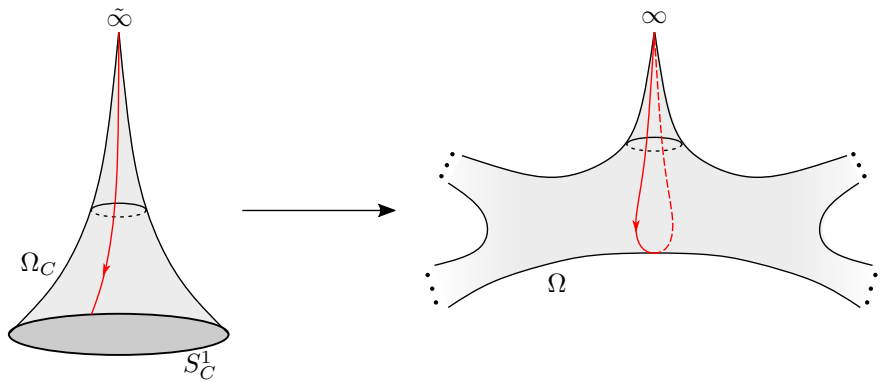

The mapping class group $\Gamma$ acts on $\Omega$ by isotopy classes of 
homeomorphisms, and such homeomorphisms all have canonical lifts to $\Omega_C$, by
lifting the action in the obvious way near infinity. Although these
maps are not quasi-isometries in the hyperbolic metric, 
they nevertheless extend {\em continuously} to homeomorphisms
of the conical circle, and these extensions depend only on the mapping class. We thus
obtain a representation $\Gamma \to \Homeorm^+(S^1_C)$.

The conical circle and this action are introduced in \cite{Bavard_Walker_1}, although
another closely related action was introduced earlier in \cite{Calegari_planar}.

\section{The simple set}

A geodesic ray $\gamma$ is {\em simple} if it is embedded in $\Omega$. Such simple geodesics
fall into three natural classes:
\begin{enumerate}
\item{Proper simple geodesics that run from $\infty$ to a point in the Cantor set; 
these are called {\em short rays}.}
\item{Proper simple geodesics that run from $\infty$ to $\infty$; these are called
{\em lassos}.}
\item{Simple geodesics that start at $\infty$ and enter $\Omega$ but are not
proper (and therefore recur somewhere); these are called {\em long rays}.}
\end{enumerate}

We denote the set of short rays by $R$, the set of lassos by $L$ and the set of long
rays by $X$. All our geodesics are {\em oriented}; for a short or long ray
this is unambiguous, since $\infty$ is the `initial point'. But for a lasso we need to
choose which end of the geodesic is `initial' and which end is `terminal'. We think
of $R,L,X$ as rays in $\Omega$, and at the same time as points in $S^1_C$. The
{\em simple set} is the union $R \cup L \cup X \subset S^1_C$.

Note that $\Gamma$ acts transitively on the set of short rays, and also on the set of
lassos; i.e.\/ $R$ and $L$ are two orbits of $\Gamma$. Also note that $R$ is uncountable,
whereas $L$ is countable.

We first prove some elementary facts about the dynamics of $\Gamma$.

\begin{lemma}[\cite{Bavard_Walker_1}]
The simple set is closed.
\end{lemma}
\begin{proof}
Let $\gamma$ be a geodesic ray from infinity. If it intersects itself, the first point
of intersection is stable under a small perturbation; thus the complement of the simple
set is open. 
\end{proof}

\begin{lemma}
$R$ is contained in the closure of every orbit. Thus the closure of $R$ is the unique
minimal set.
\end{lemma}
\begin{proof}
We can separate the Cantor set from infinity by two simple closed geodesics
$\alpha_1$, $\alpha_2$ in many different ways. For instance, let each $\alpha_j$
separate a proper subset $K_j$ of the Cantor set from infinity, so that
$K_1 \cap K_2$ is nonempty. Any geodesic ray $\gamma$ must be simple
up to the first time it crosses one of the $\alpha_j$; without loss of generality
let it cross $\alpha_1$. Now we can take a sequence of mapping classes $\phi_n$ which
shrink $K_1$ homothetically down to some point $k$ in the Cantor set and
the images $\phi_n(\gamma)$ converge to a short ray from $\infty$ to $k$. See Figure \ref{fig: shrink}.
\end{proof}

\begin{figure}
	\labellist
	\small 
	
	\pinlabel $\infty$ at 150 150
	\pinlabel $k$ at 17 50
	\pinlabel $K_1$ at 85 20
	\pinlabel $\gamma$ at 50 100	
	\pinlabel $\alpha_1$ at 10 10
	\pinlabel $K_2$ at 185 65
	\pinlabel $\alpha_2$ at 250 10
	
	\pinlabel $\infty$ at 463 150
	\pinlabel $k$ at 330 50
	\pinlabel $\phi(K_1)$ at 360 15
	\pinlabel $\phi(\gamma)$ at 360 100	
	\pinlabel $\phi(\alpha_1)$ at 300 30
	\pinlabel $\phi(K_2)$ at 460 65
	\pinlabel $\phi(\alpha_2)$ at 575 10
	\endlabellist
	\centering
	\includegraphics[scale=0.7]{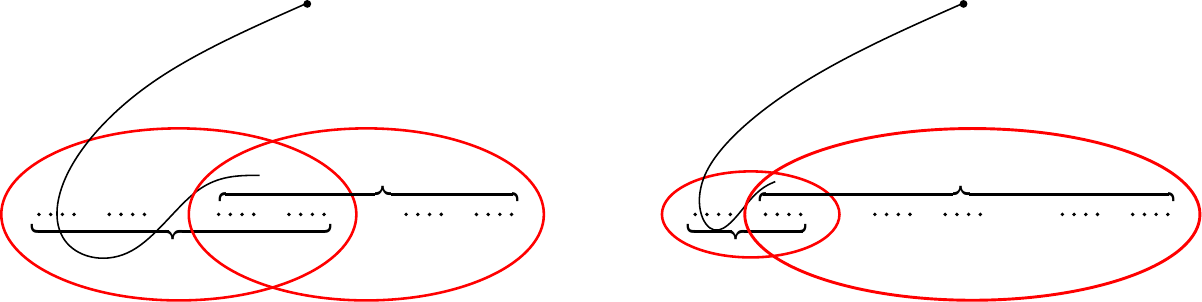}
	\caption{A homeomorphism $\phi$ shrinking Cantor subset $K_1$ and expanding $K_2$ such that $\phi^n$ shrinks $K_1$ to $k$ and $\phi^n(\gamma$) converges to a short ray to $k$}\label{fig: shrink}
\end{figure}
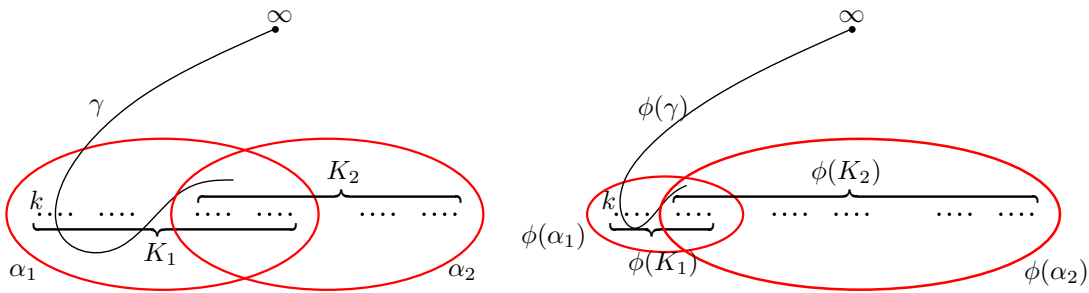

\begin{lemma}
The simple set is nowhere dense.
\end{lemma}
\begin{proof}
Any proper simple geodesic is a short ray or a lasso. We have seen that every short ray
is a limit of every orbit; in particular, a short ray can be approximated by non-simple 
geodesics. A small perturbation of a lasso is a geodesic running deep into the cusp
near infinity; any such geodesic which does not limit to infinity but goes sufficiently
deep must self-intersect. It follows that lassos are {\em isolated} in the simple set.
Finally, a non-proper geodesic $\gamma$ must recur somewhere, and therefore comes 
arbitrarily close to itself infinitely often; 
thus it may be perturbed very slightly to produce a self-intersection.
\end{proof}

We are now able to prove the main theorem of this section. 

\begin{thm}[Conical Circle]\label{theorem:conical_action}
The set $R \cup X$ is a Cantor set in $S^1_C$ and is the unique minimal set for the
action of $\Gamma$ on $S^1_C$. Every complementary interval contains a unique lasso,
and all lassos arise this way. The set of boundary points of $R\cup X$ is exactly 
the set of long rays that spiral around some simple geodesic loop in $\Omega$.
\end{thm}
\begin{proof}

\begin{figure}
	\labellist
	\small 
	
	\pinlabel $\infty$ at 145 130
	\pinlabel $p$ at 70 60
	\pinlabel $\gamma$ at 95 100
	\pinlabel $\infty$ at 458 130
	\endlabellist
	\centering
	\includegraphics[scale=0.75]{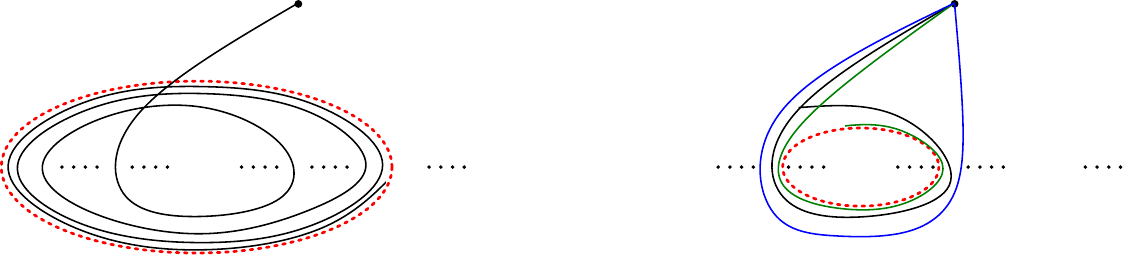}
	\caption{The left figure depicts a non-simple geodesic $\gamma$ and its first self-intersection $p$. 
		The right figure shows the perturbations of $\gamma$ in two directions to a lasso and a long ray 
		spiraling around the unique geodesic loop isotopic to the `loop' part of $\gamma$ up to $p$}\label{fig: perturb}
\end{figure}
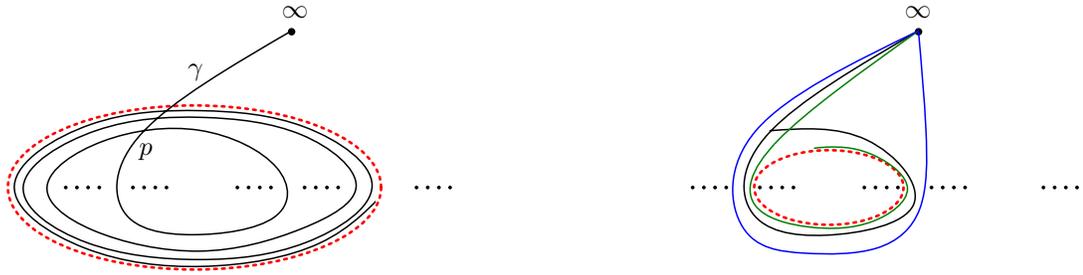

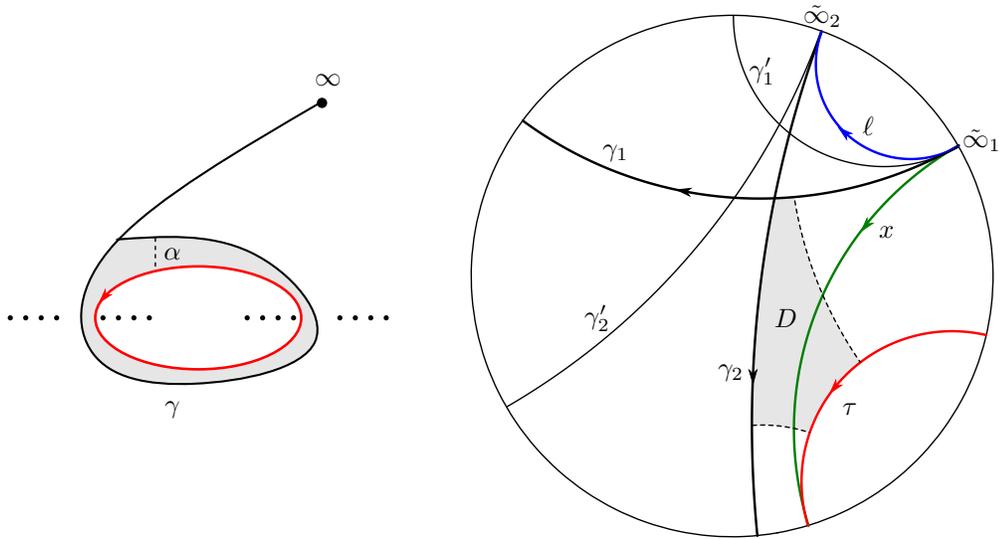
\begin{figure}
	\labellist
	\small 
	
	\pinlabel $\infty$ at 175 250
	\pinlabel $\gamma$ at 90 70
	\pinlabel $\alpha$ at 90 155
	
	\pinlabel $\tilde{\infty}_2$ at 445 285
	\pinlabel $\tilde{\infty}_1$ at 532 217
	\pinlabel $\tau$ at 460 70
	\pinlabel $D$ at 425 120
	\pinlabel $\gamma_2$ at 395 90
	\pinlabel $\gamma_1$ at 332 210
	\pinlabel $\gamma'_2$ at 322 120
	\pinlabel $\gamma'_1$ at 412 255
	\pinlabel $\ell$ at 470 225
	\pinlabel $x$ at 480 167
	\endlabellist
	\centering
	\includegraphics[scale=0.7]{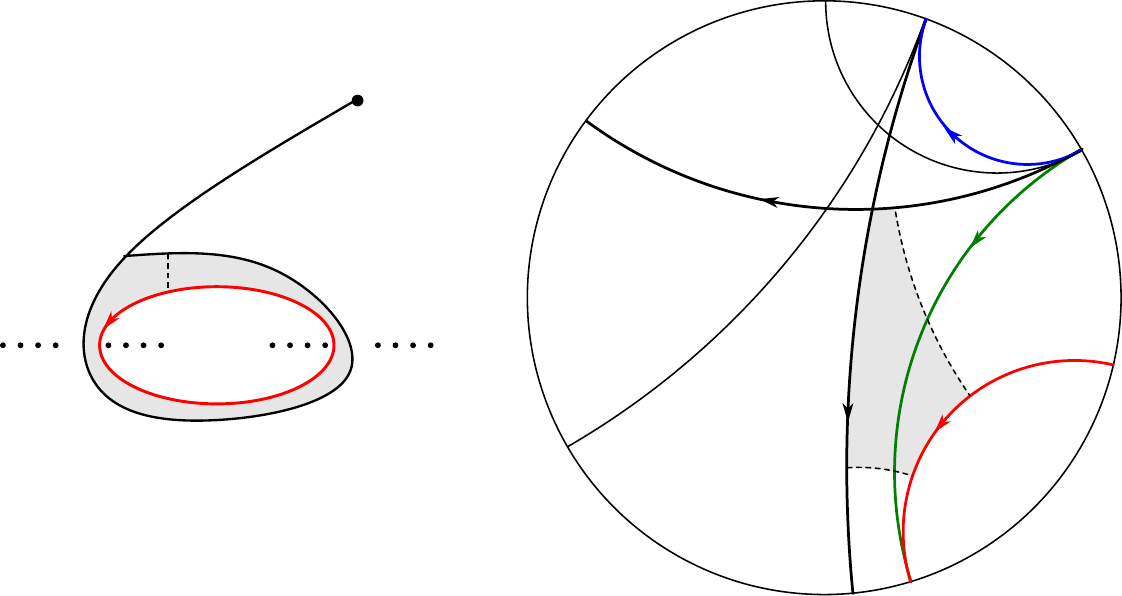}
	\caption{Pushing $\gamma_1$ to its right and left produces geodesics $\ell$ and $x$ projecting 
		to a lasso and a long spiral in $\Omega$. All smaller perturbations are not simple in $\Omega$.}\label{fig: perturbhyp}
\end{figure}

Let $\gamma$ be a non-simple geodesic. Since it begins as an embedded geodesic ray from
$\infty$ it has a first point $p$ such that the sub-geodesic from $\infty$ to $p$ self-intersects.
The geodesic up to the point $p$ looks like the letter $\rho$ embedded in $\Omega$; 
see Figure \ref{fig: perturb}. As we perturb $\gamma$, this first point of intersection
slides up and down the `stem' of the $\rho$. We can push the intersection point
in one direction all the way to $\infty$ and obtain a lasso as the limit. 
Or we can push in the other direction around and around so that
it spirals onto the unique simple geodesic which is in the isotopy class of the `loop' of the $\rho$. 
See the right part of Figure \ref{fig: perturb} and Figure \ref{fig: complementary}.
Thus every lasso is a boundary point on both sides of a complementary interval to the simple set, and 
all complementary intervals arise this way; and the other boundary point of a complementary interval is
a long ray which spirals around a simple geodesic. Let's call this kind of long ray a {\em long spiral}.

Formally, one can see this pushing on the universal cover, depicted in Figure \ref{fig: perturbhyp}.
In $\Omega$, there is an annular region between the `loop' of the $\rho$ and the geodesic loop homotopic to it. 
Pick an arbitrary geodesic arc $\alpha$ to cut this region open and lift it to the universal cover. Then the lifted
region $D$ witnesses a lift $\tau$ of the geodesic loop and two lifts $\gamma_1,\gamma_2$ of $\gamma$. Moreover, the
hyperbolic element $g_\tau$ representing the image of $\tau$ in $\Omega$ oriented as in the figure takes $\gamma_1$ to
$\gamma_2$. Then the geodesics in blue and green are obtained by pushing $\gamma_1$ to its right and left, 
projecting to the lasso and long spiral respectively. 
To see that any smaller perturbation $\gamma'_1$ of $\gamma_1$ has non-simple image in $\Omega$, note that 
$g_\tau$ takes the sector  between $\gamma_1$ and $\gamma'_1$ to the one between $\gamma_2$ and 
$\gamma'_2\defeq g_\tau \gamma'_1$. The monotonicity of the $g_\tau$-action on the boundary $\partial \mathbb{H}^2$ implies that 
$\gamma'_1$ and $\gamma'_2$ must intersect as in the figure, and thus the projection of $\gamma'_1$ in $\Omega$ is not simple.

\begin{figure}
	\labellist
	\small 
	
	\pinlabel $\infty$ at 72 103
	\pinlabel $K_1$ at -5 20
	\pinlabel $K_2$ at 150 20
	\pinlabel $L$ at 72 10
	\endlabellist
	\centering
	\includegraphics[scale=1]{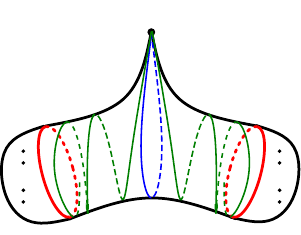}
	\caption{A lasso $L$ separating the Cantor set into $K_1\sqcup K_2$; the two long spirals corresponding
		 to the endpoints of the complementary interval of $S\cup X$ containing $L$}\label{fig: complementary}
\end{figure}
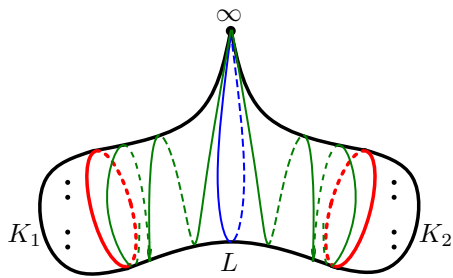

A long spiral $\gamma$ spirals around a simple geodesic $\alpha$, and this geodesic
$\alpha$ separates $\infty$ from some proper subset $K$ of the Cantor set. Pick
$k \in K$ and a simple arc from $k$ to $\alpha$, and spiral this arc around $\alpha$
in the opposite direction to the spiralling of $\gamma$; call the result $\delta$. 
We may truncate $\gamma$ after some finite amount of spiralling around $\alpha$, then
join it to $\delta$ to produce the isotopy class of a short ray. Such short rays
approximate $\gamma$ in $R \cup X$. Consequently every long spiral is in the closure of
$R$, and is not isolated in the simple set. Since no other long ray is a
boundary point of a complementary interval to the simple set, 
it follows that $R \cup X$ is a Cantor set, whose
complementary intervals each contain a unique lasso, and all lassos arise this way.

It remains to show that $R \cup X$ is minimal, equivalently that it is equal to the
closure of $R$. We have already seen that boundary points of $R\cup X$ are in the
closure of $R$; these are the long spirals. But a long ray $\gamma$ which is not a
boundary point is a limit on {\em both sides} of boundary points, which are all
limits (on one side) of short rays. Thus every long ray which is not a long spiral is a
limit of short rays from both sides, and is therefore contained in the closure of $R$.
This completes the proof.
\end{proof}

\section{Countable orbits and the simple circle}

Although $\Gamma$ is uncountable, it has several countable
orbits; the lasso set for instance. The next theorem characterizes these countable orbits.

\begin{thm}[Countable Orbit]\label{theorem:countable_orbit}
A point in $S^1_C$ has a countable orbit if and only if the associated geodesic $\gamma$ stays
inside some subsurface $\Sigma \subset \Omega$ of finite type.
\end{thm}
\begin{proof}
Suppose $\gamma$ stays in $\Sigma$. There are only countably many subsurfaces of
$\Omega$ of the topological type of $\Sigma$. Also, the mapping class group of
$\Sigma$ is countable. Thus $\gamma$ has a countable orbit.

Conversely suppose $\gamma$ goes essentially through infinitely many disjoint essential
subsurfaces $\Sigma_i$ of $\Omega$. Choose a subsequence converging to some point
$k$ in the Cantor set, and choose in each subsurface a simple loop $\alpha_i$ intersecting
$\gamma$ essentially. Since for any given initial subarc of $\gamma$ 
there is some $\Sigma_i$ disjoint from it, Dehn twists in the $\alpha_i$ perturb $\gamma$ in $S^1_C$ 
an arbitrarily small amount for $i$ large. Choose a sufficiently sparse subsequence of indices 
which we relabel $\alpha_i$, and
associated to a sequence $\sigma:\N \to \lbrace 0,1\rbrace$ we can form the mapping
class $\tau_\sigma$ which is the product of a (positive) Dehn twist in all $\alpha_i$
with $\sigma(i)=1$. If for a fixed metric $d$ on $S^1_C$ and for each $i>0$ we choose $\alpha_{i+1}$ so that 
$d(\tau_{\bar{\sigma}}(\gamma),\tau_{\sigma}(\gamma))$ is much smaller than all $d(\tau_{\sigma'}(\gamma),\tau_{\sigma}(\gamma))$
for any $\sigma\neq\sigma'$ supported on $\{1,\ldots,i\}$ and $\bar{\sigma}(k)=\sigma(k)$ iff $k\neq i+1$,
then the map $\lbrace 0,1\rbrace^\N \to S^1_C$ sending $\sigma$
to $\tau_\sigma(\gamma)$ is injective. Figure \ref{fig: uncountable_orbit} gives an illustration.
\end{proof}
\begin{figure}
	\labellist
	\small 
	
	\pinlabel $\infty$ at 115 205
	\pinlabel $\gamma$ at 103 90
	\pinlabel $\Sigma_1$ at 78 150
	\pinlabel $\alpha_1$ at 150 145
	\pinlabel $\Sigma_2$ at 140 80
	\pinlabel $\alpha_2$ at 65 75
	\pinlabel $\Sigma_3$ at 8 20
	\pinlabel $\alpha_3$ at 98 20
	\endlabellist
	\centering
	\includegraphics[scale=0.8]{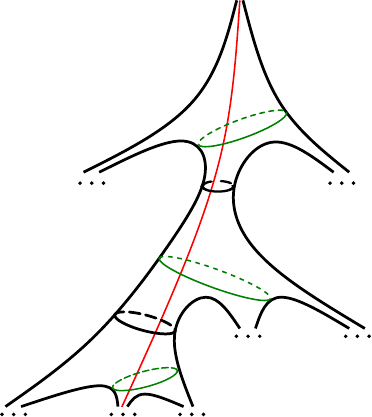}
	\caption{Infinitely many simple essential loops $\alpha_i$ in infinitely many disjoint subsurfaces $\Sigma_i$ that a ray $\gamma$ visits, with only the first three depicted}\label{fig: uncountable_orbit}
\end{figure}
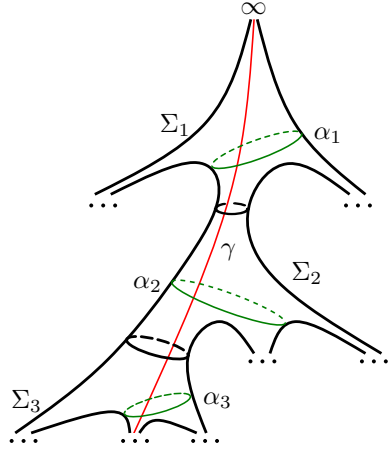

\begin{exmp}\label{example:countable_orbits}
There are uncountably many long rays with countable orbits. Pick a subsurface $\Sigma$
of finite type, and let $\Lambda$ be a minimal geodesic lamination in $\Sigma$. Since
$\Lambda$ is minimal it avoids a neighborhood of $\infty$. Let $U$ be the complementary
region to $\Lambda$ in $\Sigma$ containing infinity, and let $\lambda$ be a boundary leaf
of $U$. Then we can take $\gamma$ to be a geodesic ray from $\infty$ to one of the endpoints
of $\lambda$.
\end{exmp}

The {\em simple circle} $S^1_S$ is the quotient of the conical circle $S^1_C$ where we
collapse every complementary interval to $R\cup X$ to a point. The action of $\Gamma$
on $S^1_S$ is minimal, by Theorem~\ref{theorem:conical_action} and has uncountably
many countable orbits by Example~\ref{example:countable_orbits}. The existence of these
orbits can be used to obstruct the existence of injective homomorphisms from certain
(necessarily uncountable) circularly orderable groups to $\Gamma$.

\section{Subgroups and non-subgroups}

\begin{thm}\label{thm: countable subgroup}
A countable group $G$ is isomorphic to a subgroup of $\Gamma$ if and only if it is
circularly orderable.
\end{thm}
\begin{proof}
Since $\Gamma$ acts faithfully on $S^1$ it is circularly orderable, and so is every subgroup,
countable or not. If $G$ is circularly orderable, then it acts faithfully on $S^1$ by
homeomorphisms. Take a countable dense set of orbits of $G$ and blow them up to complementary
intervals. This defines an action of $G$ on the circle preserving a Cantor set, so that the
action on the Cantor set is already faithful. Suspend this to an action on $S^2$ where the
Cantor set is contained in the equator, and puncture $S^2$ at the north pole. This defines
a map from $G$ to $\Gamma$. This map is injective; in fact it is already injective when
we pass to the natural quotient $\Gamma \to \Homeorm(\text{Cantor})$.
\end{proof}

The argument actually shows that any action of $G$ on $S^1$ is semiconjugate to some
embedding $G \to \Gamma \to \Homeorm^+(S^1_S)$.
In other words, we can realize {\em every semiconjugacy class} of action of any countable group
on $S^1$ as a subgroup of $\Gamma$ acting on $S^1_S$.

\begin{thm}\label{thm: circle embeds}
The discrete group $S^1$ embeds in $\Gamma$. 
\end{thm}
\begin{proof}
By thinking of $\R$ as a vector space over $\Q$, we obtain
$S^1\cong \Q/\Z\oplus(\oplus_{\R/\Q} \Q)$.
It follows that $S^1$ is a subgroup of $\Q/\Z\oplus\prod_{\N}\Q$. 
The $\Q/\Z$ factor embeds into $\Gamma$ by the construction above,
preserving a Cantor set contained in the equator.

For this action, choose distinct $\Q/\Z$-orbits $Z_i$ corresponding to 
$i\in\N$ on the equator minus Cantor set such that their union only 
accumulates to points in the Cantor set. For each $i\in \N$, 
blow up each point $z\in Z_i$ to a small closed disk $D_z$ 
identified with the unit disk so that $\Q/\Z$ permutes these disks and acts 
by the identity map under the identifications. 

Fix a Cantor set in the interior of the unit disk, 
which provides a Cantor set in each $D_z$ by the given identification. 
Then the union of these Cantor sets over all $z\in Z_i$ and $i\in\N$ 
together with the original Cantor set on the equator is still a Cantor set.

Let $\Gamma_D$ denote the mapping class group of the disk minus a Cantor set
(fixed on the boundary). Now, $\Gamma_D$ is a central $\Z$ extension of $\Gamma$,
where the center is generated by a Dehn twist around the boundary.
The analog of this fact in the world of small mapping class groups is
well-known, and the same proof works here; see e.g.\/ \cite[equation (4.2)]{Farb_Margalit}. 
Thus we may embed $\Q$ into the
$\Gamma_D$ by taking the preimage of some $\Q/\Z$ in $\Gamma$.
Using such an embedding, we let the $\Q$ factor corresponding to $i\in \N$ 
act on each $D_z$ minus its Cantor set for each $z\in Z_i$ simultaneously 
under the given identification. Then this action commutes with the $\Q/\Z$ action, 
and different $\Q$ factors have disjoint supports for their actions. 
This embeds $S^1$ in $\Gamma$.
\end{proof}

In the next theorem and in \S~\ref{rigidity_section}, we will need to make use of a
well-known homological argument, whose rudiments we now explain. For a reference
see Ghys \cite{Ghys} or see \cite{Calegari_scl} for more on quasimorphisms and
stable commutator length.

Let $G$ be any group
and $\rho:G \to \Homeorm^+(S^1)$ any action on the circle. Associated to this action
there is an {\em Euler class} $e_\rho \in H^2(G;\Z)$ or just $e$ if $\rho$ is understood,
which is the pullback under $\rho$ of
the generator of $H^2(\Homeorm^+(S^1);\Z)=\Z$. The class $e$ is the obstruction to
lifting the action to $\Homeorm^+(\R)$ under the covering $\R \to S^1$.

For any $g \in \Homeorm^+(S^1)$, Poincar\'e's {\em rotation number} $\rot(g) \in \R/\Z$
is defined by choosing any lift $\tilde{g} \in \Homeorm^+(\R)$ and computing 
$$\rot^\sim(\tilde{g}): = \lim_{n \to \infty} \frac {\tilde{g}^n(0)} {n}$$
and then defining $\rot(g)$ to be the reduction of $\rot^\sim$ mod $\Z$.

If $\tilde{\rho}:G \to \Homeorm^+(\R)$ is obtained by lifting $\rho:G \to \Homeorm^+(S^1)$
the composition $\rot^\sim\circ \tilde{\rho}:G \to \R$ is a {\em homogeneous quasimorphism}
with defect $\le 1$; thinking of it as a 1-cocycle, its coboundary determines an element
in bounded cohomology $\eurm_\R \in H^2_b(G;\R)$, the {\em bounded Euler class}. In fact,
$\eurm_\R$ comes from a class $\eurm_\Z \in H^2_b(G;\Z)$ by change of coefficient, unique if $G$ is perfect (and otherwise
determined by the values of the rotation numbers on generators for $H_1(G)$).

For any group, Ghys shows that the class $\eurm_\Z \in H^2_b(G;\Z)$ determines the action of 
$G$ up to monotone equivalence --- i.e.\/ the relation generated by semiconjugacy. In
particular, the rotation numbers are determined by the class $\eurm_\Z$. For a perfect
group, the rotation numbers are determined by $\eurm_\R$; otherwise they are determined
up to an element of $\Hom(G,S^1)$.

Now, suppose that $G$ is uniformly perfect; i.e.\/ that every element is the product of
a uniformly bounded number of commutators. Bavard duality (see \cite{Calegari_scl})
implies that every homogeneous quasimorphism must vanish identically on $G$. In particular,
this implies that the comparison maps $H^2_b(G;\R) \to H^2(G;\R)$ and $H^2_b(G;\Z) \to H^2(G;\Z)$
are injective. So for a uniformly perfect group $G$ acting on the circle, 
the rotation numbers of every element are determined by the Euler class $e$. If $G$ is
a uniformly perfect group with $H^2(G;\Z)=0$ one can say further that for the lifted
action to $\R$ one has $\rot^\sim\circ\tilde{\rho}(g)=0$ for all $g$. In particular, every
orbit on $\R$ is bounded (in fact, it has diameter $\le 1$) and consequently there is
a global fixed point in $\R$ which projects to a global fixed point for the original
action in $S^1$. We summarize these observations in a proposition, which we emphasize
is essentially due to Ghys:

\begin{prop}[Ghys]\label{Euler_prop}
Let $G$ be a uniformly perfect group. If $\rho:G \to \Homeorm^+(S^1)$ is any action,
the rotation numbers of every element are determined by the Euler class $e \in H^2(G;\Z)$.
If further $H^2(G;\Z)=0$ then every $\rho:G \to \Homeorm^+(S^1)$ has a global fixed point.
\end{prop}

We are now in a position to prove the following:

\begin{thm}
$\PSL_2(\R)$ is not isomorphic to a subgroup of $\Gamma$.
\end{thm}
\begin{proof}
We will show that any faithful action of $\PSL_2(\R)$ on the circle is transitive.
Since the action of $\Gamma$ on the simple circle has countable orbits, this will 
prove the theorem.

First, observe that $\PSL_2(\R)$ is uniformly perfect. In fact, every element is a
commutator. This can be seen by hyperbolic geometry: there is a hyperbolic structure
on a once-punctured torus with a boundary geodesic of any length (including zero)
and there is also a hyperbolic cone structure on a torus with one cone point of any
angle $<2\pi$. It follows from
Proposition~\ref{Euler_prop} that for any $\rho:\PSL_2(\R) \to \Homeorm^+(S^1)$ the Euler class
$e$ determines the rotation numbers of every element.

Second, it is well-known that $H^2_b(\PSL_2(\R);\Z)=\Z$ (\cite[Example 2.13]{Matsumoto_Morita}). Since $\PSL_2(\R)$ has torsion,
no faithful action on the circle can have trivial Euler class, since otherwise the
action would lift to $\R$. Thus $e_\rho=\lambda e$ for some $\lambda\neq 0$, where $e_\rho$ 
and $e$ are the Euler classes associated to the $\rho$-action and the standard action respectively.
Certain subgroups (e.g. triangle subgroups) of $\PSL_2(\R)$
are rigid --- the only non-trivial homomorphisms to $\Homeorm^+(S^1)$ are standard. 
Thus $\lambda=\pm 1$, i.e.
there is only one Euler class $e$ (up to sign) 
associated to a faithful action of $\PSL_2(\R)$ on the circle, namely the one coming from the 
standard action on the boundary of the hyperbolic plane.

This implies that the rotation number of $\rho(g)$ is equal
to the (usual) rotation number for every $g \in \PSL_2(\R)$. In particular for $\theta$ in
the $S^1$ subgroup, the rotation number of any $\rho(\theta)$ is $\theta$. It follows that
$S^1$ acts {\em freely} on $S^1$, since any homeomorphism with a fixed point has zero rotation
number. Thus if we pick any $p \in S^1$ the orbit map $S^1 \to S^1$ is injective and 
order-preserving. We claim this orbit map is surjective. First, its image is dense; for
if not, the closure has countably many gaps, and these gaps correspond to a distinguished
countable subset of $S^1$ invariant under the action of $S^1$ on itself, which is absurd.
Second, its image is everything. To see this, let $q$ be a point in $S^1$; by considering
points to the left and right of $q$ in the image of the orbit map, we see that $q$ determines
a unique Dedekind cut and therefore a unique point in $S^1$, whose image under the
orbit map necessarily is to the right of the left points and to the left of the right points,
and is therefore equal to $q$. Since the orbit map $S^1 \to S^1$ is both 
order preserving and a bijection, it is a homeomorphism; it follows that the $S^1$ subgroup
of $\PSL_2(\R)$ must act on $S^1$ conjugate to the standard action. In particular, the
action is {\em transitive} on $S^1$, as claimed.
\end{proof}

Proposition~\ref{Euler_prop} is used also in the proof of Lemma~\ref{lemma: Gamma_(R) fix points}.

\section{Rigidity of $\Gamma$ actions on $S^1$}\label{rigidity_section}
We now arrive at the main result of the paper, the classification of $\Gamma$ actions
on the circle.
\begin{thm}[Rigidity]\label{thm: rigidity}
Any homomorphism $\Gamma \to \Homeorm^+(S^1)$ is either trivial, or is semiconjugate
(possibly up to a change of orientation) to the action on the simple circle $S^1_S$.
\end{thm}

\begin{corr}
	Any nontrivial action of $\Gamma$ on $S^1$ must contain a countable orbit and an uncountable orbit.
\end{corr}

Recall the notation $R$ for the set of short rays, i.e.\/ proper isotopy classes of simple
lines in $\Omega$ from infinity to some point on the Cantor set. If $r \in R$ is a
ray, denote by $\rayend(r)$ the endpoint of $r$ in the Cantor set.

We define two subgroups of $\Gamma$ associated to a ray.
For any ray $r\in R$, let $\Gamma_r$ be the stabilizer of $r$ in $\Gamma$ 
and let $\Gamma_{(r)}\le \Gamma_r$ be the subgroup of mapping classes 
that can be represented by homeomorphisms that are the identity in a neighborhood of 
$\rayend(r)$ (equivalently: in a neighborhood of $r$).

\begin{defn}
	We say a disk $D$ in $S^2$ is a \emph{proper dividing disk} if 
	\begin{enumerate}
		\item $\infty$ is in the exterior of $D$;
		\item both the interior and exterior of $D$ intersect the Cantor set while its boundary does not.
	\end{enumerate}
\end{defn}

\begin{lemma}\label{lemma: Gamma_(r) uniformly perfect and acyclic}
	Every element $g\in \Gamma_{(r)}$ is a commutator. In addition, $\Gamma_{(r)}$ 
is acyclic; i.e.\/ $\widetilde{H}_k(\Gamma_{(r)};\Z)=0$ for all $k$.
\end{lemma}
\begin{proof}

This is proved by Mather's suspension argument \cite{Mather}. We explain.

	Any $g\in \Gamma_{(r)}$ is supported in some proper dividing disk $D_0$ whose complement contains $r$. Then there is some $h\in \Gamma_{(r)}$ such that the sequence of disks $D_n\defeq h^n(D_0)$ are pairwise disjoint and converging to a point. See Figure \ref{fig: suspension}. It follows that $g_n=h^n g h^{-n}$ is supported in $D_n$ and any $g_i,g_j$ commute due to disjoint supports. Moreover, the infinite product $a(g)\defeq \prod_{n=0}^{\infty} g_n$ is a well-defined element in $\Gamma_{(r)}$. Then $g=a(g)ha(g)^{-1}h^{-1}$ is a commutator.
	
	Mather's suspension argument \cite{Mather} shows that $\Gamma_{(r)}$ is acyclic. We give a sketch. Each homology class $[\sigma]\in H_k(\Gamma_{(r)};\Z)$ is represented by some cycle $\sigma$ in the bar complex for $B\Gamma_{(r)}$, which only involves finitely many elements in $\Gamma_{(r)}$. Hence there is some disk $D_0$ as above containing the support of all these elements. Replacing each element $g$ involved in the cycle $\sigma$ by $a(g)$ defined using some $h\in \Gamma_{(r)}$ as above results in another cycle $a(\sigma)$. Then one can easily check that $[\sigma]=[a(\sigma)]-h_* [a(\sigma)] h_*^{-1}$ which must vanish since any inner automorphism induces the identity map on group homology.
\end{proof}

\begin{figure}
\labellist
\small 

\pinlabel $\infty$ at 130 125
\pinlabel $r$ at 200 40
\pinlabel $D_0$ at 8 0
\pinlabel $D_1$ at 88 5
\pinlabel $D_2$ at 125 8
\pinlabel $p$ at 145 10
\pinlabel $h$ at 74 47
\pinlabel $h$ at 116 42
\endlabellist
\centering
\includegraphics[scale=1]{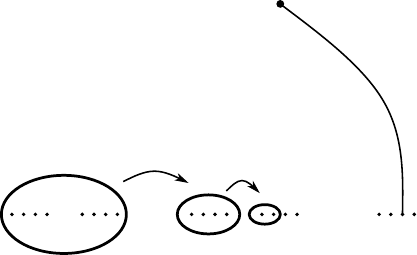}
\caption{A sequence of disjoint disks $D_n$ converging to some point $p$ in the Cantor set such that $h(D_n)=D_{n+1}$ for some element $h\in \Gamma_{(r)}$}\label{fig: suspension}
\end{figure}
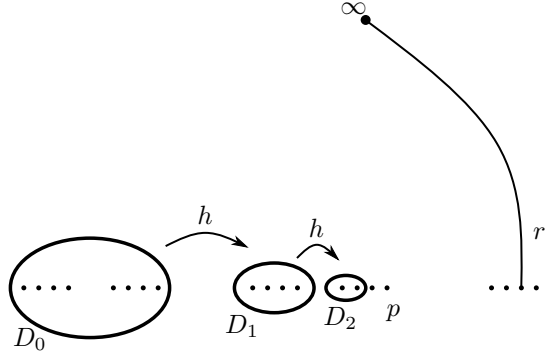

\begin{lemma}\label{lemma: Gamma_(R) fix points}
Any action of $\Gamma_{(r)}$ on a circle has a global fixed point.
\end{lemma}
\begin{proof}
We have shown that $H^2(\Gamma_{(r)})=0$ and every element of $\Gamma_{(r)}$ is a commutator;
thus this is a special case of Proposition~\ref{Euler_prop}.
\end{proof}

\begin{lemma}\label{lemma: gen Gamma_r from Gamma_(r)}
	$\Gamma_r$ is generated by $\Gamma_{(r)}$ and $\Gamma_{(s)}\cap\Gamma_r$ for any ray $s$ with
	$\rayend(s)\ne\rayend(r)$.
\end{lemma}
\begin{proof}
	Simply note that any $g\in \Gamma_r$ can be decomposed as $g=g_1 g_2$ with $g_1\in \Gamma_{(r)}$ and $g_2$ supported in any small neighborhood of the endpoint of $r$.
\end{proof}

\begin{lemma}
	For any action of $\Gamma$ on the circle, every $\Gamma_r$ has a global fixed point.
\end{lemma}
\begin{proof}
	First note that $\Gamma_{(r)}$ is normal in $\Gamma_r$ (actually, $\Gamma_r$ is
	the normalizer of $\Gamma_{(r)}$). 
	In fact, if $g\in \Gamma_{(r)}$ is the identity in some neighborhood $U$ of 
	$\rayend(r)$ and $h\in \Gamma_r$, then $hgh^{-1}$ is the identity on $h(U)$, which is 
	also a neighborhood of $\rayend(r)$. 
	
	Fix any ray $s$ with $\rayend(s) \ne \rayend(r)$. 
	Let $F_r$ and $F_s$ be the set of points on the circle fixed by $\Gamma_{(r)}$ and $\Gamma_{(s)}$ respectively. 
	These are nonempty closed sets by Lemma~\ref{lemma: Gamma_(R) fix points}. Lemma~\ref{lemma: gen Gamma_r from Gamma_(r)} implies that $\Gamma_r$ fixes any point in $F_r\cap F_s$. So we are left with the case where $F_r$ and $F_s$ are disjoint. $\Gamma_r$ permutes $F_r$ and thus its complementary intervals since $\Gamma_{(r)}$ is normal in $\Gamma_r$. It follows that $\Gamma_{(s)}\cap \Gamma_r$ fixes the boundary points of any complementary interval of $F_r$ that intersects $F_s$. Hence such boundary points would be fixed by $\Gamma_r$.
\end{proof}

\begin{lemma}\label{lemma: gen Gamma by Gamma_r}
	If $r,s\in R$ have distinct endpoints, then $\Gamma_r$ and $\Gamma_s$ generate $\Gamma$.
\end{lemma}
\begin{proof}
	Denote by $\Gamma(r,s)$ the subgroup generated by $\Gamma_r$ and $\Gamma_s$. 
	It suffices to show that $\Gamma(r,s)$ acts transitively on $R$. Since $r$ and $s$ have distinct endpoints, there is a mapping class $g$ supported in a small neighborhood of $\rayend(r)$ so that $g\in \Gamma_s$ and $r'\defeq gr$ is disjoint from $r$ (see Figure \ref{fig: localchange}). Then $\Gamma_{r'}=g\Gamma_r g^{-1}\subset \Gamma(r,s)$ and thus $\Gamma(r,r')\subset \Gamma(r,s)$. So it suffices to prove the lemma with the additional assumption that $r$ and $s$ are disjoint.
	
	Now we show that any short ray $x$ can be taken to $s$ by elements in $\Gamma(r,s)$. 
	This will evidently complete the proof.
	We show this by induction on the length of a minimal sequence
	$x=x_0,x_1,\cdots,x_n=s$ where adjacent rays are disjoint in the interior (i.e. they
	may share the endpoint). 
	Such a minimal number exists by the connectedness of the Ray
	Graph; see e.g.\/ \cite{Calegari_blog} or \cite{Bavard_hyperbolic} for details.
	
	Note that every $\Gamma_r$ is transitive on the set of rays disjoint from $r$.
	Thus if $x$ is disjoint from $s$, we can apply an element of $\Gamma_s$ to make it
	disjoint from both $r$ and $s$, and then apply an element of $\Gamma_r$ to move
	it to $s$. This proves the base case since any $x$ disjoint from $s$ in the interior
	can be made disjoint from $s$ by applying an element of $\Gamma_r$ supported in a 
	small neighborhood of $\rayend(s)$ similar to the one in Figure \ref{fig: localchange}.
	
	Now we suppose we have a sequence $x=x_0,x_1,\cdots,x_n = s$ and suppose
	$g \in \Gamma(r,s)$ takes $x_1$ to $s$. Then $gx$ is disjoint from $s$ in the interior, 
	so by the base case there is an $h \in \Gamma(r,s)$ with $hgx=s$. This completes the proof.
\end{proof}

\begin{figure}
	\labellist
	\small 
	
	\pinlabel $\infty$ at 130 127
	\pinlabel $s$ at 16 50
	\pinlabel $D$ at 86 5
	\pinlabel $r$ at 98 42
	\pinlabel $r'$ at 112 43
	\endlabellist
	\centering
	\includegraphics[scale=1]{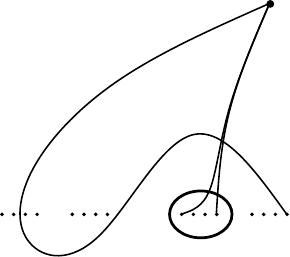}
	\caption{A mapping class supported in a small disk $D$ taking $r$ to a disjoint ray $r'$}\label{fig: localchange}
\end{figure}
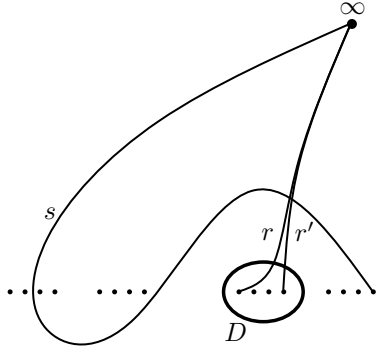

\begin{prop}\label{prop: gen Gamma by torsion}
	$\Gamma$ is generated by $2$-torsion.
\end{prop}
\begin{proof}
	Let $\Gamma_2$ be the subgroup generated by $2$-torsion. Then $\Gamma_2$ is normal in $\Gamma$. In particular, any commutator $[a,b]=a(ba^{-1}b^{-1})\in \Gamma_2$ if $a\in \Gamma_2$.
	
	Construct a Cantor set $K$ in the plane in the following manner.
	
	Take a closed disk $D_0$ with small radius centered at a point on the positive side of the $x$-axis. Let $h_0$ be a shrinking dilation centered at the origin such that $D_1\defeq h_0 D_0$ is disjoint from $D_0$. Put a Cantor set $K_0$ in the interior of $D_0$ and let $D_n\defeq h_0^n D_0$, $K_n\defeq h_0^n K_0$. Let $r$ be the rotation by $\pi$ centered at the origin and $D'_n\defeq r D_n$, $K'_n\defeq r K_n$. Then the union of all $K_n,K'_n$ and the origin is a Cantor set $K$. See Figure \ref{fig: genbytor}.
	
	In this way, the homeomorphisms $h_0$ and $r$ commute and $r^2=id$. One can modify $h_0$ to a homeomorphism $h$ that preserves $K$, still commutes with $r$, and agrees with $h_0$ on $D_n,D'_n$ for all $n\ge 1$. Considered $r$ and $h$ as elements in $\Gamma$. There is another $b\in \Gamma$ satisfying the following properties:
	\begin{enumerate}
		\item $b|_{D_n}=r$ for all $n\ge 1$;
		\item $b|_{D'_n}=rh=hr$ for all $n\ge 1$;
		\item $b|_{D_0}=id$;
		\item $b(K'_0)=K'_0\cup K_1$,
	\end{enumerate}
	where the last two bullets ensure that $b$ preserves the Cantor set $K$.
	
	Let $g\in \Gamma$ be any element supported in a proper dividing disk $D$, which we may assume to be $D_1$ by conjugation. Then $h^n g h^{-n}$ is supported in $D_{n+1}$ and the infinite product $x_g=\prod_{n\ge 0} h^n g h^{-n}\in \Gamma$ is well-defined. 
	Now $$a_g\defeq[x_g,r]=\prod_{n\ge 0} h^n g h^{-n}\cdot \prod_{n\ge 0} rh^n g^{-1} h^{-n} r^{-1}$$ is an element in $\Gamma_2$ since $r$ is $2$-torsion.
	
	Note that $bh^n g h^{-n}b^{-1}=rh^n g h^{-n}r^{-1}$ since $b|_{D_{n+1}}=r$ for all $n\ge 0$. Similarly $brh^n g h^{-n}r^{-1}b^{-1}=h^{n+1} g h^{-(n+1)}$ since $b|_{D'_{n+1}}=rh$ for all $n\ge 0$. Thus
	$$ba_gb^{-1}=\prod_{n\ge 0} bh^n g h^{-n}b^{-1}\cdot \prod_{n\ge 0} brh^n g^{-1} h^{-n} r^{-1}b^{-1}=\prod_{n\ge 0} rh^n g h^{-n}r^{-1} \cdot \prod_{n\ge 0} h^{n+1}g^{-1}h^{-(n+1)},$$
	and $g=a_g (ba_g b^{-1})$ lies in $\Gamma_2$ since $a_g$ does.
	
	In particular, any $g\in \Gamma_{(r)}$ for any ray $r\in R$ is of the form above. Hence $\Gamma_{(r)}$ is contained in $\Gamma_2$. Applying this to $r$ and $s$ with distinct endpoints, Lemma \ref{lemma: gen Gamma_r from Gamma_(r)} implies $\Gamma_r\subset \Gamma_2$. Then we conclude $\Gamma_2=\Gamma$ by Lemma \ref{lemma: gen Gamma by Gamma_r}.
\end{proof}

In the proof above, we actually proved the stronger statement that $\Gamma$ is \emph{normally} generated by a single $2$-torsion. 
In an up-comping paper \cite{Calegari_Chen}, we show that $\Gamma$ is the normally generated by any element that does not fix the Cantor set pointwise.

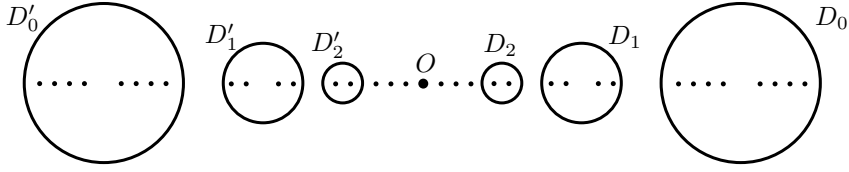
\begin{figure}
	\labellist
	\small 
	
	\pinlabel $O$ at 152 38
	\pinlabel $D'_0$ at 0 55
	\pinlabel $D'_1$ at 75 48
	\pinlabel $D'_2$ at 115 45
	\pinlabel $D_0$ at 305 55
	\pinlabel $D_1$ at 227 48
	\pinlabel $D_2$ at 180 45
	\endlabellist
	\centering
	\includegraphics[scale=1]{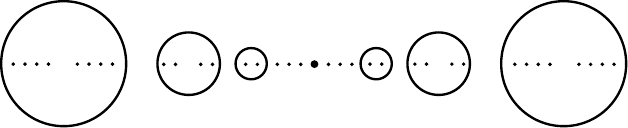}
	\caption{A specific symmetric configuration of the Cantor set and disks $D_n$, $D'_n$}\label{fig: genbytor}
\end{figure}

\begin{corr}
	Any homomorphism $\rho: \Gamma\to \Homeorm^+([0,1])$ is trivial. Thus a nontrivial homomorphism $\rho: \Gamma\to \Homeorm^+(S^1)$ has no global fixed point.
\end{corr}
\begin{proof}
	The first assertion immediately follows from Proposition \ref{prop: gen Gamma by torsion} by noticing that the group $\Homeorm^+([0,1])$ is torsion-free. Any action on the circle with a global fixed point gives rise to an action on the interval by cutting at some fixed point.
\end{proof}

From now on, fix an arbitrary action $\rho: \Gamma\to \Homeorm^+(S^1)$ without global fixed points. Fix some $r_0\in R$ and let $P(r_0)\in S^1$ be a fixed point of $\Gamma_{r_0}$. Any $r\in R$ can be written as $r=gr_0$ for some $g\in \Gamma$. Let $P(r)\defeq g P(r_0)$, which does not depend on the choice of $g$ and is stabilized by $\Gamma_r$. Then $P:R\to S^1$ is $\Gamma$-equivariant.

\begin{lemma}\label{corr: distinct image}
	If $r,s\in R$ have distinct endpoints, then $P(r)$ and $P(s)$ are distinct.
\end{lemma}
\begin{proof}
	If $P(r)=P(s)$ then its stabilizer contains both $\Gamma_r$ and $\Gamma_s$ which generate $\Gamma$ by Lemma \ref{lemma: gen Gamma by Gamma_r}. This contradicts our assumption that $\rho$ has no global fixed points.
\end{proof}

\begin{lemma}\label{lemma: p preserves order on disjoint triples}
	By appropriately choosing the orientation of $S^1$ on which $\rho(\Gamma)$ acts, the triple $(P(r_1),P(r_2),P(r_3))$ is positively oriented whenever $(r_1,r_2,r_3)$ is, for any disjoint $r_1,r_2,r_3\in R$ with distinct endpoints.
\end{lemma}
\begin{proof}
	Pick some positively oriented triple $(r_1,r_2,r_3)$ as above. Then $P(r_1),P(r_2),P(r_3)$ are distinct by Lemma \ref{corr: distinct image}. Choose the orientation on $S^1$ to make $(P(r_1),P(r_2),P(r_3))$ positively oriented. For any other such triple $(r'_1,r'_2,r'_3)$ with positive orientation, there is some $g\in \Gamma$ such that $gr_i=r'_i$ for $i=1,2,3$. Then $$(P(r'_1),P(r'_2),P(r'_3))=(gP(r_1),gP(r_2),gP(r_3))$$ is also positively oriented.
	
\end{proof}

In the sequel, fix the orientation provided by Lemma \ref{lemma: p preserves order on disjoint triples}. 

We now describe an increasing filtration of $R$ by Cantor sets. This filtration depends
on the choice of an embedded circle $\gamma$ in the plane containing the Cantor set; we
call such a $\gamma$ an {\em equator}. Observe that an equator inherits a canonical 
orientation as the boundary of the complementary region containing infinity.

For any equator $\gamma$ we define the {\em $\gamma$-length} of a short ray $r$ to be
the minimal number of transverse intersections of $r$ with $\gamma$. 
Note that we do not count $\rayend(r)\in\gamma$ as an intersection, and 
that the $\gamma$-length might be infinite. Let $R_n(\gamma) \subset R$ denote the set of
short rays with $\gamma$-length $\le n$. Thus, for example, $R_0(\gamma)$ consists of
rays going straight from infinity to some point on the Cantor set without crossing $\gamma$.
The canonical orientation on $\gamma$ induces a cyclic ordering on $R_0(\gamma)$, which
agrees with the circular order on $R$ induced by its inclusion in the simple circle.

\begin{corr}\label{corr: p preserves order, simple version}
	For any set $T$ of mutually disjoint rays with distinct endpoints, the cyclic order on $T$ agrees with
	the cyclic order on $P(T)$. In particular, this holds with $T=R_0(\gamma)$ for equator $\gamma$ as above.
\end{corr}
\begin{proof}
	A cyclic order on a set is determined by its restriction to all triples. By Lemma 
	\ref{lemma: p preserves order on disjoint triples}, the cyclic orders on $T$ and $P(T)$ agree on triples, 
	so they are equal.
\end{proof}

\begin{lemma}
	For any equator $\gamma$ and for any finite $n$ the set $R_n(\gamma)$ 
	is a Cantor set in both the conical circle and the simple circle.
\end{lemma}
\begin{proof}
	We prove the lemma in the conical circle by showing that $R_n(\gamma)$ is a perfect subset of the Cantor set 
	$R\cup X$. First, note that $R_n(\gamma)$ is closed since long rays intersect $\gamma$ infinitely many
	times, and since having at least $n+1$ transverse intersections with $\gamma$ is an open condition. Second, a 
	short ray $r$ of $\gamma$-length $n$ is a limit of short rays of the same length. To see this, note that 
	the intersections with $\gamma$ cut $r$ into $n+1$ essential arcs on alternating sides of 
	$\gamma$. Perturbing the last arc by moving	$\rayend(r)$ to nearby points in the Cantor set 
	along $\gamma$ produces a sequence of short rays with the same $\gamma$-length 
	converging to $r$; see Figure \ref{fig: limit_same_len}. Now, the quotient map 
	from the conical circle to the simple circle is injective on $R$ and therefore 
	the image of the compact subset $R_n(\gamma)$ in the simple circle is also a Cantor set.
\end{proof}

\begin{figure}
	\labellist
	\small 
	
	\pinlabel $\infty$ at 158 165
	\pinlabel $\gamma$ at 8 165
	\pinlabel $r$ at 195 18
	\endlabellist
	\centering
	\includegraphics[scale=0.5]{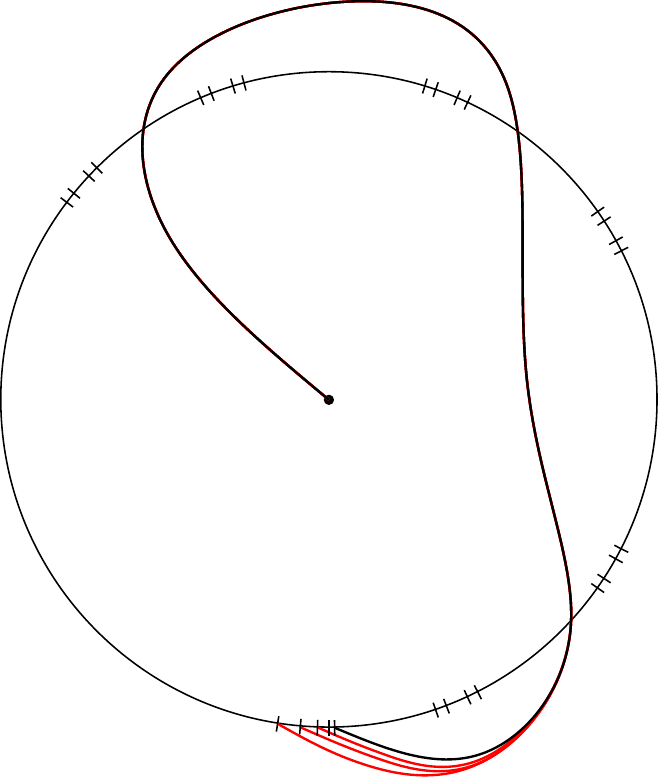}
	\caption{A sequence of short rays of the same $\gamma$-length limiting to $r$}\label{fig: limit_same_len}
\end{figure}

\begin{lemma}
	Any ray of $\gamma$-length $n$ is the limit on both sides of rays of $\gamma$-length $n+1$.
\end{lemma}
\begin{proof}
	Let $r$ be any ray of $\gamma$-length $n$. Let $\alpha_n$ be the segment of $r$ connecting its last 
	transverse intersection with $\gamma$ to $\rayend(r)$. The rays we present below can be straightened
	to geodesics up to isotopy, which does not affect the intersection patterns since geodesics minimize
	intersection numbers.
	
	If $\rayend(r)$ is not the boundary of any complementary interval of the Cantor set in $\gamma$, then 
	there are sequences of complementary intervals limiting to $\rayend(r)$ from both sides. In this case,
	we can replace $\alpha_n$ by an arc that runs through a nearby complementary interval $I$ and then limits to
	$\rayend(r)$ without further intersecting $\gamma$; see the blue rays in Figure 
	\ref{fig: limit_longer_len}. Letting $I$ limit to $\rayend(r)$ from the left or
	right creates sequences of rays of $\gamma$-length $n+1$ limiting to $r$ on both sides.
	
	If $\rayend(r)$ is the boundary of some complementary interval $I$, then the construction above can be 
	done only from the side not witnessing $I$. For the other side, we can replace $\alpha_n$ by an arc 
	running through $I$ and then ending on a point $p$ near $\rayend(r)$ in the Cantor set without further
	intersecting $\gamma$; see the red rays in Figure~\ref{fig: limit_longer_len}. Letting $p$ 
	converge to $\rayend(r)$, we obtain the desired sequence of rays of $\gamma$-length $n+1$.
	
	\begin{figure}
		\labellist
		\small 
		
		\pinlabel $\infty$ at 158 165
		\pinlabel $\gamma$ at 8 165
		\pinlabel $r$ at 192 18
		\endlabellist
		\centering
		\includegraphics[scale=0.5]{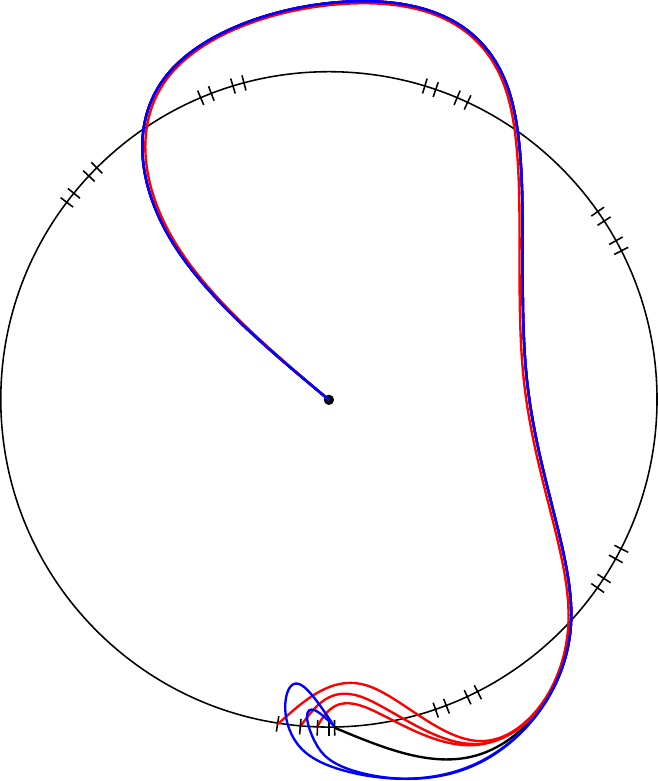}
		\caption{Two sequences of short rays, in blue and red, of $\gamma$-length $4$ limiting to $r$,
			which has $\gamma$-length $3$, from two sides}\label{fig: limit_longer_len}
	\end{figure}
\end{proof}

\begin{corr}\label{corr: length of r+-}
	For any complementary interval $(r_-,r_+)$ of $R_n(\gamma)$ in the simple circle, the rays $r_-$ and 
	$r_+$ have $\gamma$-length exactly $n$.
\end{corr}

It follows that if $r$ has $\gamma$-length $>n$ then there are two rays $r^n_-,r^n_+$ 
of $\gamma$-length $n$ associated to $r$ so that $(r^n_-,r^n_+)$ is the complementary 
interval of $R_n(\gamma)$ containing $r$, and which satisfy the following properties:
\begin{enumerate}
	\item $(r^n_-,r,r^n_+)$ is positively oriented on the simple circle; and
	\item $(r^n_-,s,r^n_+)$ is negatively oriented for any $s\in R_n(\gamma)$ other than $r^n_{\pm}$.
\end{enumerate}

\begin{lemma}\label{lemma: same first intersection}
	For any $n\ge0$ and any complementary interval $(r_-,r_+)$ of $R_n(\gamma)$, there is a complementary
	interval $I$ of the Cantor set in $\gamma$ such that all rays $r\in(r_-,r_+)$ first intersect 
	with $\gamma$ at some point in $I$.
\end{lemma}
\begin{proof}
	When $n=0$, we know from our concrete description of $R_0(\gamma)$ that $\rayend(r_-)$ and 
	$\rayend(r_+)$ are the boundary points of some complementary interval $I$ of the Cantor set in $\gamma$,
	and the rays in $(r_-,r_+)$ are exactly those that first intersect with $\gamma$ at some point in $I$.
	
	If $n>0$, by Corollary \ref{corr: length of r+-}, $r_{\pm}$ has length $n\ge1$ and intersects $\gamma$
	transversely. We claim $r_-$ and $r_+$ first intersect $\gamma$ in the same complementary interval $I$ 
	of the Cantor set in $\gamma$. If not, let $p_-$ and $p_+$ be the first intersection of $r_-$ and $r_+$ 
	with $\gamma$ respectively. Then the arc on $\gamma$ from $p_-$ to $p_+$ in the positive orientation contains some 
	point in the Cantor set, which determines a short ray $r_m\in R_0(\gamma)$. The initial arcs on $r_+,r_-$
	up to $p_+,p_-$ together with the arc on $\gamma$ from $p_-$ to $p_+$ in the positive orientation form 
	a sector, which contains $r_m$. This shows that $r_m\in (r_-,r_+)$, 
	contradicting the fact that $(r_-,r_+)$ is a complementary interval of $R_n(\gamma)$. This proves the 
	claim, from which it follows that all rays in $(r_-,r_+)$ are forced to first intersect the interval $I$ 
	as well.
\end{proof}

\begin{lemma}\label{lemma: good circle}
	For any finite collection of short rays $r_1,\ldots, r_n$ with distinct endpoints, there is 
	an equator $\gamma$ containing the Cantor set such that all $r_i$ simultaneously 
	have finite $\gamma$-length.
\end{lemma}
\begin{proof}
	Start with an arbitrary equator $\gamma$. We will modify $\gamma$ so 
	that $r_1$ has finite $\gamma$-length without changing the $\gamma$-lengths of $r_i$ for all $i>1$. Let $D$
	be a small proper dividing disk that contains a neighborhood of $\rayend(r_1)$ and is disjoint from $r_i$
	for all $i>1$. Put $r_1$ in general position with $\gamma$ and $\partial D$. Then $r_1\cap D$ has 
	finitely many components, only one of which limits to $\rayend(r_1)$. Shrinking $D$ by 
	removing finitely many bigons, one at a time, we may assume there are no other 
	components of $r_1\cap D$. Then $r_1\cap D$ is a ray in $D$ minus a Cantor set 
	starting from a marked point on the boundary. The mapping class group 
	of the disk minus a Cantor set acts transitively on the set of such rays for the same reason that 
	$\Gamma$ acts transitively on $R$; see e.g. \cite[Lemma~2.6.2]{Bavard_Walker_1}.
	Now, $\gamma\cap D$ has 
	finitely many components, one of which contains $\rayend(r_1)$; see Figure \ref{fig: good_circle}. 
	Thus there is some ray $r$ in $D$ disjoint from $\gamma\cap D$ and some $g\in \Gamma$ supported in 
	$D$ such that $g(r)=r_1\cap D$. Then $r_1\cap D$ is disjoint from $g(\gamma)\cap D$. Replacing each 
	component of $\gamma\cap D$ by its corresponding component in $g(\gamma)\cap D$, we obtain the desired 
	modification of $\gamma$ that makes the $\gamma$-length of $r_1$ finite.
\end{proof}

\begin{figure}
	\labellist
	\small 
	
	\pinlabel $\gamma$ at 40 65
	\pinlabel $r_1$ at 192 25
	\pinlabel $r$ at 200 45
	\pinlabel $D$ at 150 10
	\endlabellist
	\centering
	\includegraphics[scale=1]{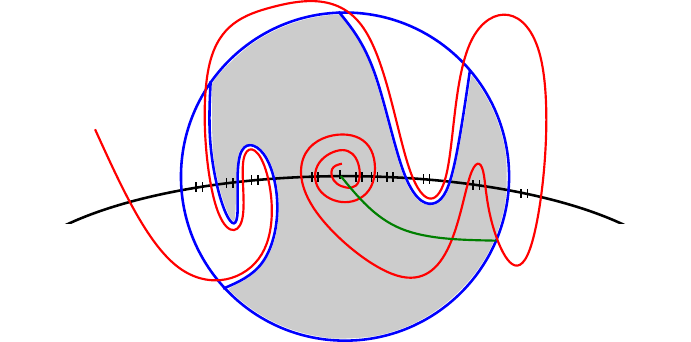}
	\caption{We shrink the proper dividing (round) disk to the smaller shaded disk $D$ such that $D\cap r_1$ 
		has only one component. There is a ray $r$ in $D$ starting from the same point in $\partial D$ as 
		$D\cap r_1$ and (whose interior is) disjoint from $\gamma$.}\label{fig: good_circle}
\end{figure}
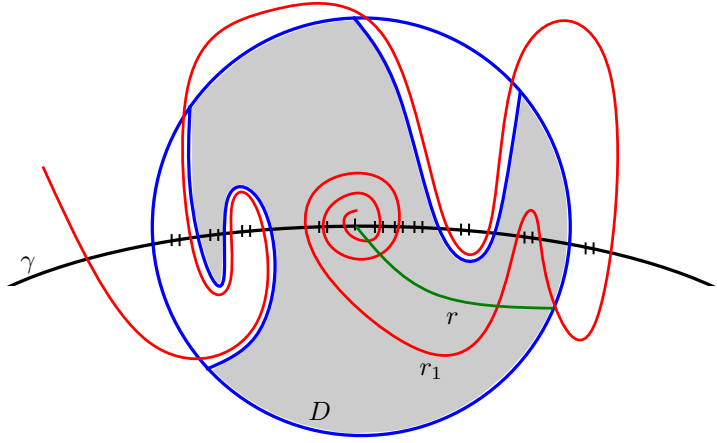

\begin{lemma}\label{lemma: p preserves order}
	The map $P: R\to S^1$ preserves the circular order. In particular, it is injective.
\end{lemma}
\begin{proof}
	We first induct on $n\ge 0$ to prove that $P$ is order-preserving on $R_n(\gamma)$ for all $\gamma$ as
	above. The base case is covered by Corollary \ref{corr: p preserves order, simple version}. Suppose
	$n\ge0$ and $P$ is order-preserving on $R_n(\gamma)$ for all $\gamma$. For any $\gamma$, it suffices to 
	show that $P$ is order-preserving on $R_{n+1}(\gamma)\cap [r_-,r_+]$
	for each complementary interval $(r_-,r_+)$ of $R_n(\gamma)$. Let $I$ be the complementary interval 
	of the Cantor set in $\gamma$ provided by Lemma \ref{lemma: same first intersection}. Let $\alpha_-$
	and $\alpha_+$ be the geodesic short rays in $R_0(\gamma)$ going directly from infinity to the boundary 
	points of $I$. Consider the triangular region $\Delta$ bounded by sides $I$, $\alpha_-$ and $\alpha_+$. 
	For any $r\in R_{n+1}(\gamma)\cap [r_-,r_+]$ represented by a geodesic, let $\alpha_r$ be the subarc of 
	$r$ from infinity to its first intersection with $\gamma$. Then $r$ never enters $\Delta$ from 
	$\alpha_-$ and leaves from $\alpha_+$, or vice versa, since otherwise this would result in an 
	intersection with $\alpha_r$ and thus a self-intersection of $r$. Thus every component of $r\cap \Delta$
	has at least one end on $I$. Isotoping $I$ through $\Delta$ to $\alpha_-\cup \alpha_+$ and slightly 
	further across infinity to an arc $I'$, we obtain a new equator $\gamma'$ such that any ray 
	$r\in R_{n+1}(\gamma)\cap [r_-,r_+]$ has $\gamma'$-length at most $n$ by resolving the first 
	intersection with $\gamma$ without creating new intersections; see Figure \ref{fig: push}.
	Thus by the induction hypothesis, $P$ is order-preserving on this set of short rays, which completes our 
	induction.
	
	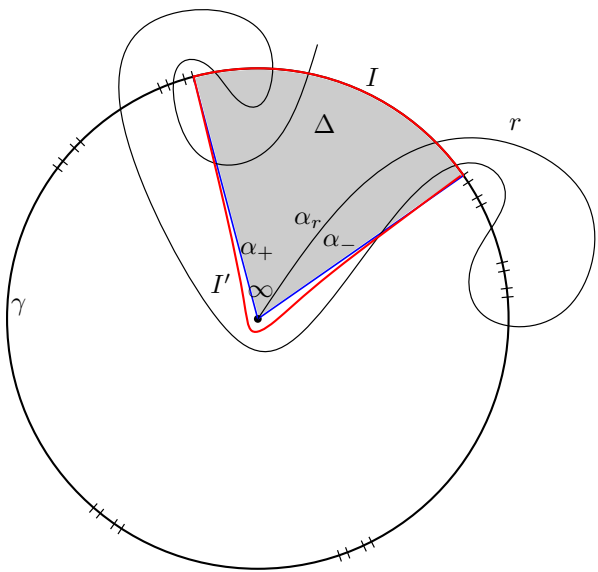
\begin{figure}
		\labellist
		\small 
		
		\pinlabel $\infty$ at 160 175
		\pinlabel $\alpha_+$ at 158 200
		\pinlabel $I'$ at 135 180
		\pinlabel $\alpha_r$ at 190 220
		\pinlabel $\alpha_-$ at 210 205
		\pinlabel $I$ at 230 310
		\pinlabel $\Delta$ at 200 280
		\pinlabel $\gamma$ at 8 165
		\pinlabel $r$ at 320 280
		\endlabellist
		\centering
		\includegraphics[scale=0.6]{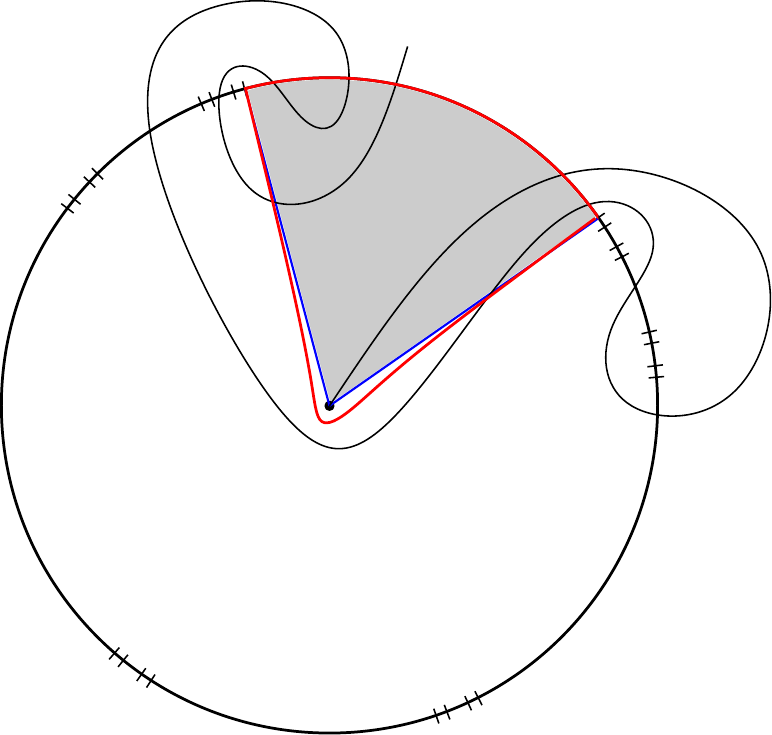}
		\caption{Pushing $I$ across the triangle $\Delta$ changes the equator $\gamma$ to $\gamma'$ reducing the length of $r$}\label{fig: push}
	\end{figure}
	
	Next, by Lemma \ref{lemma: good circle}, $P$ preserves the order of any triple of short rays with 
	distinct endpoints.
	
	Finally, consider an arbitrary positively oriented triple of rays $(r_1,r_2,r_3)$. 
	Evidently any ray $r\in R_0(\gamma)$ is the limit of disjoint rays in $R_1(\gamma)$ 
	with distinct endpoints on both sides.
	Thus by transitivity of the $\Gamma$ action on $R$, every ray $r\in R$ is the limit of disjoint rays with
	distinct endpoints on both sides in the simple circle. In particular, there are rays $r_{i+},r_{i-}$ 
	such that $(r_{1\pm},r_{2\pm},r_{3\pm})$ is a positively oriented triple of rays with distinct endpoints
	for any choice of the $\pm$ signs, and that $(r_{i-},r_i,r_{i+})$ is a positively oriented triple of 
	disjoint rays for any $i=1,2,3$. By what we have shown and Corollary 
	\ref{corr: p preserves order, simple version}, $(P(r_{1\pm}),P(r_{2\pm}),P(r_{3\pm}))$ is positively 
	oriented independent of the choice of signs and 
	$(P(r_{i-}),P(r_i),P(r_{i+}))$ is positively oriented for any $i=1,2,3$. This implies that 
	$(P(r_1),P(r_2),P(r_3))$ is positively oriented.
\end{proof}

Now we prove Theorem \ref{thm: rigidity}.
\begin{proof}[Proof of Theorem \ref{thm: rigidity}]
	Let $\hat{P}$ be the closure of	the image of $P$. Then $\hat{P}$ is $\Gamma$-invariant. 
	Collapse all complementary intervals of $\hat{P}$ to obtain another circle $S^1_P$ where $\Gamma$ acts. Then the quotient map $\pi: S^1\to S^1_P$ semi-conjugates this action to $\rho$. By Theorem \ref{theorem:conical_action}, there is always some ray in the interval $(r,s)\subset S^1_S$ for any two $r,s\in R$. Thus the image of $\pi\circ P$ is still injective and in	the correct circular order by Lemma \ref{lemma: p preserves order}. This correspondence of dense $\Gamma$-orbits extends to a homeomorphism $S^1_S\cong S^1_P$ that conjugates the two actions.
\end{proof}

\section{Acknowledgments}

We would like to thank Juliette Bavard, Ian Biringer, Kathryn Mann, Curt McMullen,
Christian Rosendal and Alden Walker for helpful comments. We also thank the anonymous 
referee for detailed suggestions improving the paper. 
Danny Calegari was partially supported by NSF grant 1405466.

\appendix

\section{Homology of big mapping class groups}

Let $\Gammahat$ denote the mapping class group of the sphere minus a Cantor set. It is closely related to $\Gamma$ by the following point-pushing Birman exact sequence
$$1\to \pi_1(\Omegahat)\to \Gamma\to \Gammahat\to 1$$
where $\Omegahat:= \Omega \cup \infty$ is the sphere minus a Cantor set.

It is shown in \cite{Calegari_blog} that $\Gammahat$ is uniformly perfect, which implies
\begin{lemma}\label{lemma: H_1}
	$H_1(\Gammahat;\Z)=0$.
\end{lemma}

The goal is to compute $H_2(\Gammahat)$ and $H^2(\Gammahat)$. All (co)homology groups in this appendix are singular (co)homology with $\Z$ coefficients.
The main theorem we prove is the following:
\begin{thm}[Homology]\label{thm: H_2}
	$H_2(\Gammahat)=\Z/2$ and $H^2(\Gammahat)=0$.
\end{thm}

The strategy is to break the classifying space $B\Gammahat$ into simpler ones by a fragmentation argument 
which dates back to Segal \cite{Segal}. See also \cite{Sergiescu_Tsuboi}. Here we realize $B\Gammahat$ as 
$G\backslash|E\Gammahat|$, where $E\Gammahat$ is the simplicial complex whose $k$-simplices are $(k+1)$-tuples
$(g_0,\ldots, g_k)\in \Gammahat^{k+1}$, and $|E\Gammahat|$ is the underlying topological space. For any 
finite subset $P$ of the Cantor set, let $\Gammahat_P$ be the subgroup of elements represented by 
homeomorphisms that are the identity in some neighborhood of $P$. Then $E\Gammahat_P$ naturally sits inside
$E\Gammahat$ as a subcomplex in a $\Gammahat_P$-equivariant way. 
Note that $E_P\defeq \Gammahat\cdot(E\Gammahat_P)$ is the subcomplex of 
$E\Gammahat$ consisting of simplices whose components $g_0,\ldots, g_k$ coincide in some neighborhood of $P$. 
Then $E_P$ is $\Gammahat$-invariant and $B\Gammahat_P=\Gammahat_P\backslash |E\Gammahat_P|=\Gammahat\backslash |E_P|\subset B\Gammahat$. 
In the sequel, we always use this preferred realization of $B\Gammahat_P$ and denote it simply by $B_P$. 

Let $B_*\defeq \cup_P B_P$, where $P$ runs over all finite subsets of the Cantor set. Note that $B_P \cap B_Q=B_{P\cup Q}$ and $B_*=\cup_p B_{\{p\}}$ where $p$ runs over the Cantor set.

\begin{lemma}\label{lemma: fragmentation}
	The inclusion $B_*\to B\Gammahat$ is a homotopy equivalence, and thus $H_k(B\Gammahat)\cong H_k(B_*)$ for any $k\ge 0$.
\end{lemma}
\begin{proof}
	This follows the argument in \cite[Proposition 4.1]{Sergiescu_Tsuboi}. Let $E_*\defeq\cup_P E_P$, 
	where $P$ runs over all finite subsets of the Cantor set. It suffices to show $|E_*|$ is contractible 
	since it has a free $\Gammahat$ action with quotient $B_*$. We will actually show that any cycle consisting 
	of simplices $\sigma_i=(g_0^{(i)},\ldots,g_k^{(i)})\subset E_{P_i}$ can be coned off in $E_*$.
	Note by definition that the components $g_j^{(i)}$ of $\sigma_i$ agree in a neighborhood of $P_i$, and thus
	$\sigma_i\subset E_{P'_i}$ for any set $P'_i$ sufficiently close to $P_i$.
	Up to replacing $P_i$ by nearby sets, we can assume that all $P_i$ are disjoint and that 
	all $g_j^{(i)}(P_i)$ are also disjoint for any fixed $j$
	(note that $g_j^{(i)}(P_i)$ does not depend on $j$).
	
	Then there are small neighborhoods $V_i$ of $P_i$ such that $V_i\cap V_j=\emptyset$ for all $i\neq j$ and the
	components $g_j^{(i)}$ of $\sigma_i$ restrict to the same homeomorphism $h_i$ on $V_i$ with 
	$h_i(V_i)\cap h_j(V_j)=\emptyset$. We may choose each $V_i$ to be a disjoint union of open disks whose 
	intersections with the Cantor set are clopen subsets of the Cantor set. Then there is a homeomorphism 
	$h:S^2\to S^2$ preserving the Cantor set such that $h|_{V_i}=h_i$. Now $\bar{\sigma_i}\defeq (h,g_0^{(i)},\ldots,g_k^{(i)})$
	(i.e. the cone over $\sigma_i$ with vertex $h$)	is in $E_{P_i}$, which cones off the cycle in $E_*$ as desired,
	that is, $\partial \sum_i \bar{\sigma_i}=\sum_i \sigma_i$.
\end{proof}

To compute the homology of $B_*$, we need to understand the homology of the subspaces $B_P$ and their unions.
\begin{lemma}\label{lemma: suspension}
	$B_{\{p\}}$ is acyclic for any $p$, that is, $\widetilde{H}_k(B_{\{p\}})=0$ for all $k$.
\end{lemma}
\begin{proof}
	We omit the proof since it is almost the same as the proof of Lemma \ref{lemma: Gamma_(r) uniformly perfect and acyclic} using Mather's suspension argument.
\end{proof}

\begin{lemma}\label{lemma: H_1 union of B_p}
	We have $H_1(\cup_{p\in P} B_{\{p\}})=0$ for any nonempty finite subset $P$ of the Cantor set.
\end{lemma}
\begin{proof}
	We induct on $|P|$. The base case $|P|=1$ is proved by Lemma \ref{lemma: suspension}. Suppose the conclusion holds for $P'$, and we consider $P=P'\cup \{q\}$ for some $q\notin P'$. Applying Mayer--Vietoris to the pair $(\cup_{p\in P'} B_{\{p\}},B_{\{q\}})$, we have a surjection $H_1(\cup_{p\in P'} B_{\{p\}})\oplus H_1(B_{\{q\}})\to H_1(\cup_{p\in P} B_{\{p\}})$. The domain is trivial by the induction hypothesis, so the conclusion follows.
\end{proof}

For each $g\in \Gammahat_P$, there are $|P|$ disjoint disks each containing one element of $P$ such that $g$ 
is the identity in these disks. Thus $g$ projects to a mapping class in $\MCG(S^2-|P|D^2)$. The disks depend
on $g$ but the groups $\MCG(S^2-|P|D^2)$ are canonically isomorphic by restricting homeomorphisms.
Thus this gives rise 
to a well-defined homomorphism $\pi: \Gammahat_P\to \MCG(S^2-|P|D^2)$. Denote the kernel by $\Gammahat_P^0$.

We say a disk in $S^2$ is a \emph{dividing disk} if both its interior and exterior intersect the Cantor set while its boundary does not.

\begin{lemma}\label{lemma: factorize}
	Any $g \in \Gammahat_P^0$ can be factored in $\Gammahat_P^0$ as a product of finitely many elements each supported in some dividing disk.
\end{lemma}
\begin{proof}
	We induct on $|P|$. When $|P|=1$, every $g \in \Gammahat_P^0$ is itself supported in a dividing disk. Suppose $|P|=n\ge 2$ and we have proved the lemma when $|P|<n$. Fix any $g \in \Gammahat_P^0$ supported in $S\defeq S^2-\cup_{i=1}^n D_i$, where $D_i$ are disjoint dividing disks each containing an element of $P$. Let $K$ be the part of Cantor set in $S$, which we may assume to be nonempty and thus is itself a Cantor set.
	
	If $g$ preserves some simple arc $\alpha\subset S-K$ connecting two distinct boundary components $\partial D_i$ and $\partial D_j$, then up to homotopy $g$ is supported in $S-N(\alpha)$ for some tubular neighborhood $N(\alpha)$ of $\alpha$ disjoint from $K$. Note that $g$ is homotopic to the identity in $S-N(\alpha)\cong S^2-(n-1)D^2$. The induction hypothesis gives the desired factorization for $g$.
	
	In the general case, it suffices to factor $g$ as the product of some $g_i\in \Gammahat_P^0$, where each $g_i$ preserves some simple arc connecting two distinct boundary components of $S$. Fix a simple arc $\alpha\subset S-K$ connecting $\partial D_1$ and $\partial D_2$. Then $g\alpha$ and $\alpha$ are homotopic in $S$ since $g$ is trivial in $\MCG(S)$. 
	
	Suppose $g\alpha$ and $\alpha$ are disjoint in $S-K$. Then they cobound a disk $D$ in $S$. If $D\cap K=\emptyset$, then $g$ preserves $\alpha$ and we are done. Otherwise, $D\cap K$ is itself a Cantor set. Pick two disjoint arcs $\beta$ and $\gamma$ in $D-K$ that cut $D$ into three subdisks $D_a,D_b,D_c$ as shown in Figure \ref{fig: factorize}. Let $K_*\defeq K\cap D_*$ with $*=a,b,c$. Then there is $h\in \Gammahat_P^0$ preserving $\gamma$ such that $h(\alpha)=\beta$.
	Indeed, such a homeomorphism $h$ can be constructed by squeezing $D_a\cup D_b$ (resp. $K_a\cup K_b$) to $D_b$ (resp. $K_b$) and expanding part of $K_c$ to the whole $K_c$ and pushing the rest of $K_c$ to $K_a$ along a path $P$ that connects $D_c$ to $D_a$ disjoint from $\gamma$. Such a path exists since $\gamma$ is homotopic to the non-separating arc $\alpha$ in $S$. By symmetry, there is also $h'\in \Gammahat_P^0$ preserving some simple arc homotopic to $\alpha$ such that $h'(g\alpha)=\beta$. Then $h''=h^{-1} h' g\in \Gammahat_P^0$ preserves the arc $\alpha$. Hence $g=h'^{-1} h h'' $ factors $g$ into elements preserving arcs as desired.
	
	Now suppose $g\alpha$ and $\alpha$ intersect. Since $g\alpha$ and $\alpha$ are homotopic in $S$, there is some innermost bigon $D$ in $S$ whose boundary consists of two arcs $\alpha_0\subset \alpha$ and $\alpha_g\subset g\alpha$. We assume $D\cap K$ is nonempty since otherwise the intersection number of $g\alpha$ and $\alpha$ in $S-K$ can be reduced. Similar to the argument above, there are $h,h'\in \Gammahat_P^0$ preserving certain arcs such that $h^{-1} h' g$ maps $\alpha$ to $\alpha'$, where $\alpha'$ is obtained from $g\alpha$ by substituting the subarc $\alpha_g$ by $\alpha_0$. Now $\alpha'$ and $\alpha$ have a smaller intersection number. Repeating this process finitely many times, we obtain a factorization of $g$ into elements in $\Gammahat_P^0$ each preserving some simple arc connecting $\partial D_1$ and $\partial D_2$.	
\end{proof}

\begin{figure}
	\labellist
	\small 
	
	\pinlabel $\alpha$ at 75 110
	\pinlabel $\beta$ at 115 110
	\pinlabel $\gamma$ at 160 110
	\pinlabel $g\alpha$ at 230 110
	\pinlabel $K_a$ at 70 60
	\pinlabel $K_b$ at 125 60
	\pinlabel $K_c$ at 200 60
	\pinlabel $K_{c,1}$ at 180 90
	\pinlabel $K_{c,2}$ at 215 90
	
	\pinlabel $\alpha$ at 365 100
	\pinlabel $h\alpha$ at 405 100
	\pinlabel $h\beta$ at 438 100
	\pinlabel $h\gamma=\gamma$ at 480 100
	\pinlabel $hg\alpha$ at 515 140
	\pinlabel $hK_{c,2}$ at 372 57
	\pinlabel $hK_a$ at 410 60
	\pinlabel $hK_b$ at 445 60
	\pinlabel $hK_{c,1}$ at 500 60
	\endlabellist
	\centering
	\includegraphics[scale=0.7]{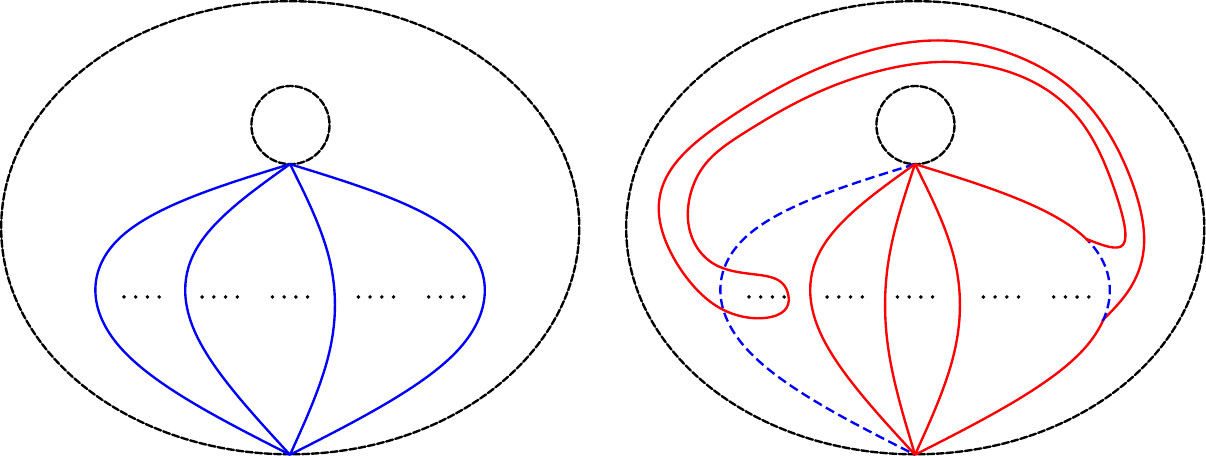}
	\caption{The figure on the left depicts the arcs $\beta,\gamma$ dividing the bigon into three subdisks. The figure on the right shows a homeomorphism $h$ preserving $\gamma$ and mapping $\alpha$ to $\beta$.}\label{fig: factorize}
\end{figure}
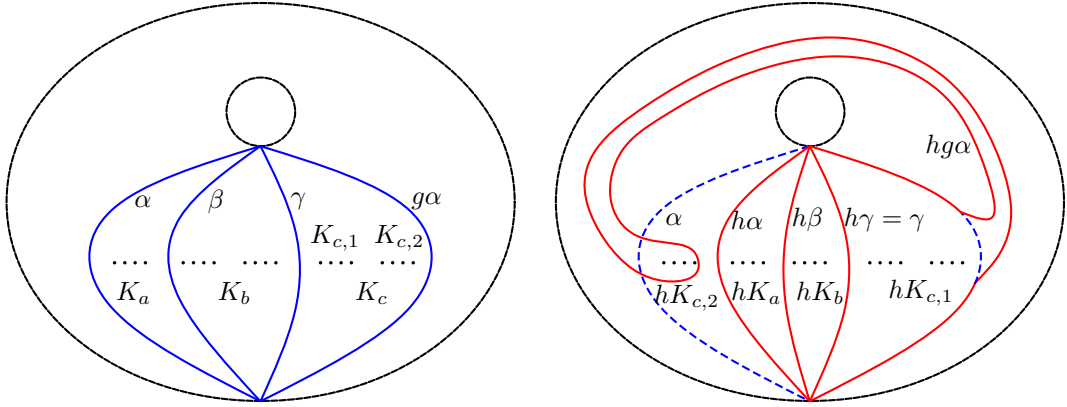
	
\begin{corr}\label{corr: H_1 iso to MCG}
	$H_1(\Gammahat_P^0)=0$ and the induced map $\pi_*: H_1(B_P)\to H_1\MCG(S^2-|P|D^2)$ is an isomorphism.
\end{corr}
\begin{proof}
	Each $h\in \Gammahat_P^0$ supported in a dividing disk is a commutator by the suspension trick. Thus $H_1(\Gammahat_P^0)=0$ by Lemma \ref{lemma: factorize}. The exact sequence
	$$0=H_1(\Gammahat_P^0)\to H_1(B_P)\stackrel{\pi_*}{\to} H_1\MCG(S^2-|P|D^2)\to 0, $$
	shows $\pi_*$ is an isomorphism.
\end{proof}

For $P=P_1\sqcup P_2$, we have an inclusion $f_{P_2}:B_P\to B_{P_1}$ by forgetting the information about $P_2$. When $P_2=\{p\}$ we simply denote $f_{\{p\}}$ by $f_p$.

\begin{lemma}\label{lemma: forgetful surjection base case}
	For any subset $P$ of the Cantor set, we have $H_1(B_P)\cong\Z$ if $|P|=2$ and $H_1(B_P)\cong\Z^3$ if $|P|=3$. Moreover, the image of $$(f_{p*},f_{q*},f_{r*}): H_1(B_{\{p,q,r\}})\to H_1(B_{\{q,r\}})\oplus H_1(B_{\{p,r\}})\oplus H_1(B_{\{p,q\}})\cong \Z^3$$ has index $2$ with basis $e_1+e_2,e_2+e_3,e_3+e_1$, where $e_1,e_2,e_3$ are the generators of the factors. In particular, any pair of maps, say $(f_{p*},f_{q*})$, is surjective.
\end{lemma}
\begin{proof}
	We have $H_1(B_P)\cong H_1\MCG(S^2-|P|D^2)$ by Corollary \ref{corr: H_1 iso to MCG}. When $|P|=2$, the computation follows from the fact that $\MCG(S^2-|P|D^2)\cong\Z$ is generated by a Dehn twist.
	
	For $S=S^2-D_1 \cup D_2 \cup D_3$, the Dehn twists around $\partial D_1$ and $\partial D_2$ generate a $\Z^2$ subgroup which gives rise to a central extension 
	$$1\to \Z^2\to \MCG(S)\to \mathrm{PB}_2\to 1,$$
	where $\mathrm{PB}_2\cong \Z$ is the pure braid group with two strands keeping track of the configuration of $D_1$ and $D_2$ in $S^2-D_3$; see e.g.\/ \cite{Farb_Margalit} for
	details. Then the Dehn twist around $\partial D_3$ hits the generator of $\mathrm{PB}_2$. It follows that these three Dehn twists generate $\MCG(S)$, and their images under $(f_{p*},f_{q*},f_{r*})$ are exactly the basis elements as claimed. This implies $H_1(B_P)\cong \MCG(S)\cong\Z^3$.
\end{proof}

\begin{lemma}\label{lemma: H_1 union of B_pq}
	Let $P$ and $\{q\}$ be disjoint subsets of the Cantor set, where $2\le |P|<\infty$. Then $H_1(\cup_{p\in P} B_{\{p,q\}})=0$.
\end{lemma}
\begin{proof}
	Fix $r\in P$ and let $P'=P-\{r\}$. We obtain the following exact sequence by applying Mayer--Vietoris to the pair $(\cup_{p\in P'} B_{\{p,q\}},B_{\{r,q\}})$ and noticing that $(\cup_{p\in P'} B_{\{p,q\}})\cap B_{\{r,q\}}=\cup_{p\in P'} B_{\{p,q,r\}}$:
	
	$$H_1(\cup_{p\in P'} B_{\{p,q,r\}}) \overset{\alpha}{\to} H_1(\cup_{p\in P'} B_{\{p,q\}}) \oplus H_1(B_{\{r,q\}}) \to H_1(\cup_{p\in P} B_{\{p,q\}})\to 0.$$
	
	It suffices to show $\alpha=(i_1,i_2)$ is surjective, which we now prove by induction on $|P|$. The base case $|P|=2$ follows from Lemma \ref{lemma: forgetful surjection base case}. When $|P|\ge 3$, we have $H_1(\cup_{p\in P'} B_{\{p,q\}})=0$ by the induction hypothesis. For any specific choice of $p_0\in P'$, the composition 
	$$H_1(B_{\{p_0,q,r\}}) \to  H_1(\cup_{p\in P'} B_{\{p,q,r\}})\overset{i_2}{\to} H_1(B_{\{r,q\}})$$
	is surjective. It follows that $i_2$ and hence $\alpha$ are surjective.
\end{proof}

\begin{lemma}\label{lemma: key H_2 coker}
	Let $P$ and $\{r\}$ be disjoint finite subsets of the Cantor set. Consider the map $(i_P)_*: H_2(\cup_{p\in P} B_{\{p,r\}})\to H_2(\cup_{p\in P} B_{\{p\}})$ induced by inclusion. Then either $\mathrm{coker}(i_P)_*=\Z/2$ for all $P$ with $2\le |P|<\infty$, or there is some $N\ge 3$ such that $\mathrm{coker}(i_P)_*=0$ for all $P$ with $N\le |P|<\infty$. 
\end{lemma}
\begin{proof}
	Fix $q\in P$ and let $P=P'\cup \{q\}$. Consider the inclusion of pairs $(\cup_{p\in P} B_{\{p,r\}},B_{\{q,r\}})\to (\cup_{p\in P} B_{\{p\}},B_{\{q\}})$. By naturality of Mayer--Vietoris, we have the following commutative diagram, where the columns are exact.

	\[	
	\begin{tikzcd}[column sep=5.4em]
	H_2(\cup_{p\in P'} B_{\{p,r\}}) \oplus H_2(B_{\{q,r\}}) 	\arrow[d] \arrow[r, "{((i_{P'})_*,(i_{\{q\}})_*)}"]	
	& H_2(\cup_{p\in P'} B_{\{p\}}) \oplus H_2(B_{\{q\}})		\arrow[d]\\
	H_2(\cup_{p\in P} B_{\{p,r\}})								\arrow[d, "\partial"] \arrow[r, "{(i_{P})_*}"] 
	& H_2(\cup_{p\in P} B_{\{p\}})								\arrow[d]\\
	H_1(\cup_{p\in P'} B_{\{p,q,r\}})							\arrow[d, "{(f_q,f_{P'})}"] \arrow[r, "f_r"]
	& H_1(\cup_{p\in P'} B_{\{p,q\}})							\arrow[d]\\
	H_1(\cup_{p\in P'} B_{\{p,r\}}) \oplus H_1(B_{\{q,r\}})		\arrow[r]
	& H_1(\cup_{p\in P'} B_{\{p\}}) \oplus H_1(B_{\{q\}})	
	\end{tikzcd}
	\]
	
	The last term on the right vanishes by Lemma \ref{lemma: H_1 union of B_p}. 
	
	When $|P|=2$, let $p$ be the unique element of $P'$. Then the first term on the right vanishes by Lemma \ref{lemma: H_1 union of B_p}, and thus $H_2(\cup_{p\in P}B_{\{p\}})\to H_1(B_{\{p,q\}})\cong \Z$ is an isomorphism. By Lemma \ref{lemma: forgetful surjection base case}, the restriction $f_r: {\rm Im} \partial\to H_1(B_{\{p,q\}})$ has image $2\Z\subset \Z$. It follows that $\mathrm{coker}(i_P)_*=\Z/2$.
	
	Suppose $|P|\ge3$. Then Lemma \ref{lemma: H_1 union of B_pq} applies to $|P'|\ge 2$, which together with Lemma \ref{lemma: suspension} reduces the diagram above to
	
	\[	
	\begin{tikzcd}[column sep=5.4em]
	H_2(\cup_{p\in P'} B_{\{p,r\}})						 		\arrow[d] \arrow[r, "{(i_{P'})_*}"]	
	& H_2(\cup_{p\in P'} B_{\{p\}}) 							\arrow[d]\\
	H_2(\cup_{p\in P} B_{\{p,r\}})								\arrow[d, "\partial"] \arrow[r, "{(i_{P})_*}"] 
	& H_2(\cup_{p\in P} B_{\{p\}})								\arrow[d]\\
	{\rm Im} \partial											\arrow[r, "f_r"]
	& 0							
	\end{tikzcd}
	\]
	
	From the diagram we see $\mathrm{coker}(i_{P'})_*$ surjects $\mathrm{coker}(i_P)_*$. Thus the conclusion follows by induction on $|P|$.
	
\end{proof}

\begin{lemma}\label{lemma: H_2 stable}
	Let $P\subset Q$ be two finite subsets of the Cantor set. Then there is some $N\ge 3$ such that the inclusion $H_2(\cup_{p\in P} B_{\{p\}})\to H_2(\cup_{p\in Q} B_{\{p\}})$ is isomorphic whenever $|P|>N$, and $H_2(\cup_{p\in P} B_{\{p\}})$ stabilizes to either $0$ or $\Z/2$.
\end{lemma}
\begin{proof}
	Consider the two possibilities in Lemma \ref{lemma: key H_2 coker}. Suppose $\mathrm{coker}(i_{P'})_*=\Z/2$ for all $P'$ with $2\le |P'|<\infty$. Let $|P|>2$ and $P=P'\sqcup\{q\}$. Note that $(\cup_{p\in P'} B_{\{p\}})\cap B_{\{q\}}=\cup_{p\in P'} B_{\{p,q\}}$, by Mayer--Vietoris, we have the following exact sequence.
	
	\[
	H_2(\cup_{p\in P'} B_{\{p,q\}})\overset{j}{\to} H_2(\cup_{p\in P'} B_{\{p\}})\oplus H_2(B_{\{q\}})\to H_2(\cup_{p\in P} B_{\{p\}})\to H_1(\cup_{p\in P'} B_{\{p,q\}}).
	\]
	
	Since $|P'|=|P|-1\ge 2$, the last term in the exact sequence above vanishes by Lemma \ref{lemma: H_1 union of B_pq}. On the other hand, Lemma \ref{lemma: suspension} and $\mathrm{coker}(i_{P'})_*=\Z/2$ implies the map $j$ has index two. Thus $H_2(\cup_{p\in P} B_{\{p\}})\cong \Z/2$. Moreover, if we further have $|P'|>2$, the exact sequence shows that the inclusion $H_2(\cup_{p\in P'} B_{\{p\}})\to H_2(\cup_{p\in P} B_{\{p\}})$ is an isomorphism. This verifies the $\Z/2$ case of our assertion with $N=2$.
	
	Now suppose	there is some $N\ge 3$ such that $\mathrm{coker}(i_{P'})_*=0$ for all $P'$ with $N\le |P'|<\infty$. Again let $|P|>N$ and $P=P'\sqcup\{q\}$. The map $j$ in the same exact sequence above is now surjective. Thus $H_2(\cup_{p\in P} B_{\{p\}})=0$ for all such $P$, which establishes the other case of our assertion.
\end{proof}
\begin{rmk}
	In contrast, a similar Mayer--Vietoris computation shows $H_2(\cup_{p\in P} B_{\{p\}})=\Z$ when $|P|=2$.
\end{rmk}

The lemmas above narrow down $H_2(B\Gammahat)$ to two possibilities, either $\Z/2$ or $0$. Lemma \ref{lemma: H_2 surj Z/2} below gives the extra information we need. Denote the Cantor subset of $S^2$ by $K$. Let $\Homeorm^+(S^2,K)$ be the group of orientation preserving homeomorphisms of $S^2$ preserving $K$. One can show that any $f\in \Homeorm^+(S^2-K)$ uniquely extends to $\bar{f}\in \Homeorm^+(S^2,K)$, and thus the restriction map gives an isomorphism $\Homeorm^+(S^2,K)\cong \Homeorm^+(S^2-K)$.

\begin{lemma}\label{lemma: H_2 surj Z/2}
	$H_2(\Gammahat)$ is isomorphic to $H_2(B\Homeorm^+(S^2,K))$, which surjects $\Z/2$ via the map $i_*:H_2(B\Homeorm^+(S^2,K))\to H_2(B\Homeorm^+(S^2))\cong \Z/2$ induced by inclusion.
\end{lemma}
\begin{proof}
	The identity component of $\Homeorm^+(S^2-K)$ is contractible by \cite[Theorem 1.1]{Yagasaki}. Thus $B\Gammahat$ and $B\Homeorm^+(S^2,K)$ are weakly homotopy equivalent, so the first assertion immediately follows. 
	
	Let $K_4=\Z/2\times \Z/2$ be the Klein $4$-group. Realize $BK_4$ as $\R P^\infty\times \R P^\infty$ and note that $H_2(K_4)\cong \Z/2$ by the K\"unneth formula, where the generator is represented by $[\R P^1\times \R P^1]$. We claim that any inclusion of a $K_4$ subgroup into $\Homeorm^+(S^2)$ induces an isomorphism $H_2(K_4)\cong H_2(B\Homeorm^+(S^2))$. This implies the second assertion by considering any $K_4$ subgroup of $\Homeorm^+(S^2,K)$ and the composition $K_4\to \Homeorm^+(S^2,K)\to \Homeorm^+(S^2)$.
	
	To prove the claim we may assume $K_4$ to be generated by two rigid rotations by angle $\pi$ with perpendicular axes, for all $K_4$ subgroups of $\Homeorm^+(S^2)$ conjugate into $\mathrm{SO}(3)$ \cite{Zimmermann}. Let $f:BK_4\to B\Homeorm^+(S^2)$ be the map induced by the inclusion. Recall that $B\Homeorm^+(S^2)\simeq B\mathrm{SO}(3)$ by reduction of structure group, and that $H_2(B\mathrm{SO}(3))\cong \pi_2(B\mathrm{SO}(3))\cong \pi_1(\mathrm{SO}(3))\cong \Z/2$ by the Hurewicz theorem and the long exact sequence of fibration. Thus the map $g:S^2\to B\mathrm{SO}(3)$ representing the nontrivial element in $\pi_2$ classifies the unique nontrivial oriented $S^2$ bundle over $S^2$. We need to show $f_*[\R P^1\times \R P^1]=g_*[S^2]$. This is actually the case since the following diagram commutes up to homotopy, i.e. the two pull back oriented $S^2$ bundles on $T^2$ are isomorphic, where $p$ is a degree-one map.
	\[	
	\begin{tikzcd}
	\R P^1\times \R P^1\cong T^2 							\arrow[d] \arrow[r, "p"]	& S^2		\arrow[d, "g"]\\
	\R P^\infty\times \R P^\infty	\arrow[r, "f"]				& B\Homeorm^+(S^2)\\
	\end{tikzcd}
	\]
\end{proof}

Finally we are in a place to prove Theorem \ref{thm: H_2}.
\begin{proof}[Proof of Theorem \ref{thm: H_2}]
	Recall that $B_*=\cup_P B_{\{p\}}$ where $p$ runs over the Cantor set. Then Lemma \ref{lemma: H_2 stable} implies that $H_2(B_*)$ is either $0$ or $\Z/2$. Combining with Lemma \ref{lemma: fragmentation} and Lemma \ref{lemma: H_2 surj Z/2}, we conclude that $H_2(\Gammahat)\cong \Z/2$. Then the desired result follows from Lemma \ref{lemma: H_1} and the universal coefficient theorem.
\end{proof}
\begin{rmk}
	The fact that $H_2(\Gammahat)\cong \Z/2$ in turn implies that $\mathrm{coker}(i_P)_*=\Z/2$ for all $P$ with $2\le |P|<\infty$ in Lemma \ref{lemma: key H_2 coker}, and that $H_2(\cup_{p\in P} B_{\{p\}})$ actually stabilizes to $\Z/2$ in Lemma \ref{lemma: H_2 stable}.
\end{rmk}

The Birman exact sequence gives rise to a Leray--Serre spectral sequence relating the
(co)-homology of $\Gamma$ and $\Gammahat$. Since $H^2(\Gammahat)=H^2(\Omegahat)=0$ it follows
that there is a left-exact sequence

$$0 \to H^2(\Gamma) \to H^1(\Gammahat,H^1(\Omegahat)) \to H^3(\Gammahat)$$
where the last map is the $d_2$ differential. We know that $H^2(\Gamma)$ contains a $\Z$
subgroup, generated by the Euler class of the action on the simple circle. 
This class is not torsion, since $\Gamma$ contains
torsion of all orders, so no finite index subgroup lifts to an action on $\R$. 

\begin{quest}
Does $H^2(\Gamma)=\Z$? 
\end{quest}

\begin{quest}
What class in $H^1(\Gammahat,H^1(\Omegahat))$ is the image of the Euler class in $H^2(\Gamma)$?
\end{quest}

\end{document}